\documentclass[reqno]{amsart}

\usepackage{amsmath,amssymb,amscd, accents} \usepackage{graphicx}
\DeclareGraphicsExtensions{.eps} \usepackage{mathrsfs}
\usepackage[mathcal]{eucal}

\newtheorem{Thm}{Theorem}{\bfseries}{\itshape}
\newtheorem*{Thm*}{Theorem}{\bfseries}{\itshape}
\newtheorem{Cor}{Corollary}{\bfseries}{\itshape}
\newtheorem{Prop}[Cor]{Proposition}{\bfseries}{\itshape}
\newtheorem{Lem}[Cor]{Lemma}{\bfseries}{\itshape}
\newtheorem*{Lem*}{Lemma}{\bfseries}{\itshape}
{\bfseries}{\itshape}
{\bfseries}{\itshape}
\newtheorem{Def}[Cor]{Definition}{\bfseries}{\rmfamily}
{\scshape}{\rmfamily}
\newtheorem{Rem}[Cor]{Remark}{\scshape}{\rmfamily}
{\bfseries}{\itshape}

\renewcommand\ge{\geqslant} \renewcommand\le{\leqslant}
\let\tildeaccent=\~ \let\hataccent=\^
\renewcommand\~[1]{\widetilde{#1}}

\def\<{\left<} \def\>{\right>} \def\({\left(} \def\){\right)}

\def\abs#1{\left\vert #1 \right\vert} \def\norm#1{\left\Vert #1
  \right\Vert} 

\let\parasymbol=\S \def\secref#1{\parasymbol\ref{#1}}
 \def\pd#1#2{\tfrac{\partial#1}{\partial#2}}

\let\polishL=l \def\Zoladek.{\.Zol\c adek}

 \def\const{\operatorname{const}}
\def\codim{\operatorname{codim}} \def\Mat{\operatorname{Mat}}
 
 \def\Im{\operatorname{Im}}
 \def\dist{\operatorname{dist}}
\def\diam{\operatorname{diam}} 
\def\GL{\operatorname{GL}} \def\SL{\operatorname{SL}}
\def\PGL{\operatorname{PGL}}
\def\etc.{\emph{etc}.}

\def\:{\colon} \def\R{{\mathbb R}} \def\C{{\mathbb C}} \def\Z{{\mathbb
    Z}} \def\N{{\mathbb N}} \def\Q{{\mathbb Q}} \def\P{{\mathbb P}}
\def\H{{\mathbb H}}

\def\A{{\mathbb A}}
\def\K{{\mathbb K}}
\def\M{{\mathbb M}}
\def\B{{\mathbb B}}

\let\PolishL=\L 
\def\L{{\mathbb L}}

 \def\e{\varepsilon} \def\S{\varSigma}

\def\poly{\operatorname{poly}}

 \def\d{\,\mathrm d}

 \def\Lojas.{\PolishL ojasiewicz}
 
\def\cP{{\mathcal P}} \def\cR{{\mathcal R}}

\def\cF{{\mathcal F}} \def\cL{{\mathcal L}} \def\cR{{\mathcal R}}

\def\cB{{\mathcal B}}
\def\cM{{\mathcal M}}

\def\cW{{\mathcal W}}

\def\cO{{\mathcal O}}

\def\Aut{\operatorname{Aut}} 
\def\slope{\sphericalangle} 
 \def\mult{\operatorname{mult}}

\def\V{V}

\def\rest#1{{\vert_{#1}}}

\def\clo{\operatorname{Clo}}
\def\vol{\operatorname{vol}}

\def\vxi{{\boldsymbol\xi}}

\def\valpha{{\boldsymbol\alpha}}

\def\vx{{\mathbf x}}
\def\vy{{\mathbf y}}
\def\vz{{\mathbf z}}
\def\vw{{\mathbf w}}

\def\vc{{\mathbf c}}

\def\vs{{\mathbf s}}
\def\vv{{\mathbf v}}

\def\id{\operatorname{id}}

\def\Qa{\Q^{\text{alg}}}

\def\CD{C^D}

\def\alg{\mathrm{alg}}
\def\RE{\mathrm{RE}}
\def\trans{\mathrm{trans}}
\def\ws{\mathrm{ws}}
\def\disc{\operatorname{disc}}

\def\hFal{h_{\mathrm{Fal}}}
\def\slope{\sphericalangle}

\newcommand{\mo}[1][k]{M^{\smash{(#1)}}}
\newcommand{\bmo}{M}

\begin{document}

\title{Point counting for foliations over number fields}

\author{Gal Binyamini}
\address{Weizmann Institute of Science, Rehovot, Israel}
\email{gal.binyamini@weizmann.ac.il}

\thanks{This research was supported by the ISRAEL SCIENCE FOUNDATION
  (grant No. 1167/17) and by funding received from the MINERVA
  Stiftung with the funds from the BMBF of the Federal Republic of
  Germany. This project has received funding from the European
  Research Council (ERC) under the European Union's Horizon 2020
  research and innovation programme (grant agreement No 802107)}


\begin{abstract}
  Let $\M$ be an affine variety equipped with a foliation, both
  defined over a number field $\K$. For an algebraic $V\subset\M$ over
  $\K$ write $\delta_V$ for the maximum of the degree and log-height
  of $V$. Write $\Sigma_V$ for the points where the leafs intersect
  $V$ improperly. Fix a compact subset $\cB$ of a leaf $\cL$. We prove
  effective bounds on the geometry of the intersection $\cB\cap V$. In
  particular when $\codim V=\dim\cL$ we prove that $\#(\cB\cap V)$ is
  bounded by a polynomial in $\delta_V$ and
  $\log\dist^{-1}(\cB,\Sigma_V)$. Using these bounds we prove a result
  on the interpolation of algebraic points in images of $\cB\cap V$ by
  an algebraic map $\Phi$. For instance under suitable conditions we
  show that $\Phi(\cB\cap V)$ contains at most $\poly(g,h)$ algebraic
  points of log-height $h$ and degree $g$.

  We deduce several results in Diophantine geometry. i) Following
  Masser-Zannier, we prove that given a pair of sections $P,Q$ of a
  non-isotrivial family of squares of elliptic curves that do not
  satisfy a constant relation, whenever $P,Q$ are simultaneously
  torsion their order of torsion is bounded effectively by a
  polynomial in $\delta_P,\delta_Q$. In particular the set of such
  simultaneous torsion points is effectively computable in polynomial
  time. ii) Following Pila, we prove that given $V\subset\C^n$ there
  is an (ineffective) upper bound, polynomial in $\delta_V$, for the
  degrees and discriminants of maximal special subvarieties. In
  particular it follows that Andr\'e-Oort for powers of the modular
  curve is decidable in polynomial time (by an algorithm depending on
  a universal, ineffective Siegel constant). iii) Following Schmidt,
  we show that our counting result implies a Galois-orbit lower bound
  for torsion points on elliptic curves of the type previously
  obtained using transcendence methods by David.
\end{abstract}
\maketitle
\date{\today}

\section{Introduction}
\label{sec:intro}

This paper is roughly divided into two parts. In~\secref{sec:intro} we
state our main technical results on point counting for
foliations. This includes upper bounds for the number of intersections
between a leaf of a foliation and an algebraic variety
(Theorem~\ref{thm:zero-count}), a corresponding bound for the covering
of such intersections by Weierstrass polydiscs
(Theorem~\ref{thm:wp-cover}), and consequently a counting result for
algebraic points in terms of height and degree
(Theorem~\ref{thm:alg-count}) in the spirit of the Pila-Wilkie theorem
and Wilkie's conjecture. The proofs of these result are given
in~\secref{sec:mult-ops}--\secref{sec:alg-count-proof}.

In the second part starting~\secref{sec:applications} we state three
applications of our point counting results in Diophantine
geometry. These include an effective form of Masser-Zannier bound for
simultaneous torsions points on squares of elliptic curves, and in
particular effective polynomial-time computability of this set; a
polynomial bound for Pila's proof of Andr\'e-Oort for $\C^n$, and in
particular the polynomial-time decidability (by an algorithm with an
ineffective constant); and a proof of Galois-orbit lower bounds for
torsion points in elliptic curves following an idea of Schmidt. We
also briefly discuss similar implications for Galois-orbit lower
bounds in Shimura varietis, to be presented in an upcoming paper with
Schmidt and Yafaev. The proofs of these results are given
in~\secref{sec:effective-mz-proof}--\secref{sec:degree-bounds-proof}.

Finally in Appendix~\ref{appendix:fuchs-growth} we prove some growth
estimates for solutions of inhomogeneous Fuchsian differential
equations over number fields. These are used in our treatment of the
Masser-Zannier result and would probably be similarly useful in many
of its generalizations.

\subsection{Setup}

In this section we introduce the main notations and terminology used
throughout the paper.

\subsubsection{The variety}

Let $\M\subset\A^N_\K$ be an irreducible affine variety defined over a
number field $\K$. We equip $\M$ with the standard Euclidean metric
from $\A^N$, denoted $\dist$, and denote by $\B_R\subset\M$ the
intersection of $\M$ with the ball of radius $R$ around the origin in
$\A^N$. Set $\B:=B_1$.

\subsubsection{The foliation}

Let $\vxi:=(\xi_1,\ldots,\xi_n)$ denote $n$ commuting, generically
linearly independent, rational vector fields on $\M$ defined over
$\K$.  We denote by $\cF$ the (singular) foliation of $\M$ generated
by $\vxi$ and by $\Sigma_\cF\subset\M$ the union of the polar loci of
$\xi_1,\ldots,\xi_n$ and the set of points where they are linearly
dependent.

For every $p\in\M\setminus\Sigma_\cF$ denote by $\cL_p$ the germ of
the leaf of $\cF$ through $p$. We have a germ of a holomorphic map
$\phi_p:(\C^n,0)\to\cL_p$ satisfying $\partial\phi_p/\partial x_i=\xi_i$ for
$i=1,\ldots,n$. We refer to this coordinate chart as the
$\vxi$-coordinates on $\cL_p$.

\subsubsection{Balls and polydiscs}

If $A\subset\C^n$ is a ball (resp. polydisc) and $\delta>0$, we denote
by $A^\delta$ the ball (resp. polydisc) with the same center, where
the radius $r$ (resp. each radii $r$) is replaced by $\delta^{-1}r$.
If $\phi_p$ continues holomorphically to a ball $B\subset\C^n$ around
the origin then we call $\cB:=\phi_p(B)$ a $\vxi$-ball. If $\phi_p$
extends to $B^\delta$ we denote $\cB^\delta:=\phi_p(B^\delta)$.

\subsubsection{Degrees and heights}

We denote by $h:\Q^\alg\to\R_{\ge0}$ the absolute logarithmic Weil
height. If $x\in\Q^\alg$ has minimal polynomial
$a_0\prod_{i=1}^d(x-x_i)$ over $\Z[x]$ then
\begin{equation}
  h(x) = \frac1d\left(\log |a_0| + \sum_{i=1}^d \log^+|x_i|\right), \qquad \log^+\alpha=\max\{\log\alpha,0\}.
\end{equation}
We also denote $H(x):=e^{h(x)}$. We define the height of a vector
$\vx\in(\Q^\alg)^n$ as the maximal height of the coordinates.

For a variety $V\subset\M$ we denote by $\deg V$ the degree with
respect to the standard projective embedding $\A^n\to\P^n$; we define
$h(V)$ as the height of the Chow coordinates of $V$ with respect to
this embedding. For a vector field $\vxi$ we define $\deg\vxi$
(resp. $h(\vxi)$) as the maximum degree (resp. logarithmic height) of
the polynomials $\vxi(\vx_i)$ where $\vx_i$ are the affine coordinates
on the ambient space. Finally we set
\begin{equation}
  \delta_\M:=\max([\K:\Q],\deg\M,h(\M))
\end{equation}
and
\begin{align}
  \delta(V) &:= \max(\delta_\M,\deg V,h(V)) & \delta(\vxi)&:=\max(\delta_\M,\deg\vxi,h(\vxi)).
\end{align} 
We sometimes write $\delta_V,\delta_\vxi$ for $\delta(V),\delta(\vxi)$
to avoid cluttering the notation.

\subsubsection{The unlikely intersection locus}

Let $V\subset\M$ be a pure-dimensional subvariety of codimension at
most $n$ defined over $\K$. We define the \emph{unlikely intersection
  locus} of $V$ and $\cF$ to be
\begin{equation}
  \Sigma_V := \Sigma_\cF\cup\{ p\in\M: \dim(V\cap\cL_p)>n-\codim V\},
\end{equation}
i.e. the set of points $p$ where $V$ intersects $\cL_p$ improperly.

\subsubsection{Weierstrass polydiscs}

Let $\cB$ be a $\vxi$-ball. We say that a coordinate system $\vx$ is a
\emph{unitary coordinate system} if it is obtained from the
$\vxi$-coordinates by a linear unitary transformation.

Let $X\subset\cB$ be an analytic subset of pure dimension $m$. We say
that a polydisc $\Delta:=\Delta_z\times\Delta_w$ in the unitary
$\vx=\vz\times\vw$ coordinates is a Weierstrass polydisc for $X$ if
$\bar\Delta\subset\cB$ and if $\dim\Delta_z=m$ and
$X\cap(\bar\Delta_z\times\partial\Delta_w)=\emptyset$. In this case
the projection $\Delta\cap X\to\Delta_z$ is a proper ramified covering
map, and we denote its (finite) degree by $e(\Delta,X)$ and call it
the degree of $X$ in $\Delta$.

\subsubsection{Asymptotic notation}

We use the asymptotic notation $Z=\poly_X(Y)$ to mean that $Z<P_X(Y)$
where $P_X$ is a polynomial depending on $X$. In this text the
coefficients of $P_X$ can always be explicitly computed from $X$
unless explicitly stated otherwise. We similarly write $Z=O_X(Y)$ for
$Z< C_X\cdot Y$ where $C_X\in\R_\ge0$ is a constant depending on $X$.

Throughout the paper the implicit constants in asymptotic notations
are assumed to depend on the ambient dimension of $\M$, which we omit
for brevity. All implicit constants are effective unless explicitly
stated otherwise (this occurs only in Theorem~\ref{thm:ao-decidable}
on Andr\'e-Oort for powers of the mdoular curve).

\subsection{Statement of the main results}
\label{sec:main-results}

Our first main theorem is the following bound for the number of
intersections between a $\vxi$-ball and an algebraic variety of
complementary dimension. Throughout this section we let $R$ denote a
positive real number.

\begin{Thm}\label{thm:zero-count}
  Suppose $\codim V=n$ and let $\cB\subset\B_R$ be a
  $\vxi$-ball of radius at most $R$. Then
  \begin{equation}
    \#(\cB^2\cap V) = \poly(\delta_\vxi,\delta_V,\log R,\log\dist^{-1}(\cB,\Sigma_V))
  \end{equation}
  where intersection points are counted with multiplicities.
\end{Thm}

The reader may for simplicity consider the case $R=1$. The general
case reduces to this case immediately by rescaling the coordinates on
$\M$ and the vector fields $\xi$ by a factor of $R$. This rescaling
factor enters logarithmically into $\delta_V$ and $\delta_\vxi$, hence
the dependence on $\log R$ in the general case. To simplify our
presentation we will therefore consider only the case $R=1$ in the
proof of Theorem~\ref{thm:zero-count}.

\begin{Rem}\label{rem:use-polydiscs}
  Similarly to the comment above, by rescaling each coordinate
  separately we may also work with arbitrary polydiscs instead of
  arbitrary balls.
\end{Rem}

We also record a corollary which is sometimes useful in the case of
higher codimensions.

\begin{Cor}\label{cor:zero-count-any-codim}
  Let $V\subset\M$ have arbitrary codimension and let
  \begin{equation}
    \Sigma := \Sigma_\cF \cup\{ p\in\M: \dim(V\cap\cL_p)>0\}.
  \end{equation}
  Let $\cB\subset\B_R$ be a $\vxi$-ball of radius at most $R$. Then
  \begin{equation}
    \#(\cB^2\cap V) = \poly(\delta_\vxi,\delta_V,\log R,\log\dist^{-1}(\cB,\Sigma))
  \end{equation}
  where intersection points are counted with multiplicities.
\end{Cor}

Our second main theorem states that the intersection between a
$\vxi$-ball and a subvariety admits a covering by Weierstrass
polydiscs of effectively bounded size.

\begin{Thm}\label{thm:wp-cover}
  Suppose $\codim V\le n$ and let $\cB\subset\B_R$ be a $\vxi$-ball of
  radius at most $R$. Then there exists a collection of Weierstrass
  polydiscs $\{\Delta_\alpha\subset\cB\}$ for $\cB\cap V$ such that
  the union of $\Delta^2_\alpha$ covers $\cB^2$ and
  \begin{equation}
    \#\{\Delta_\alpha\},\max_\alpha e(\cB\cap V,\Delta_\alpha) = \poly(\delta_\vxi,\delta_V,\log R,\log\dist^{-1}(\cB,\Sigma_V)).
  \end{equation}
\end{Thm}

The same comment on rescaling to the case $R=1$ applies to
Theorem~\ref{thm:wp-cover} as well.

\begin{Rem}
  It would also have been possible to state our results in invariant
  language for a general algebraic variety and its foliation without
  fixing an affine chart and a basis of commuting vector fields. We
  opted for the less invariant language in order to give an explicit
  description of the dependence of our constants on the foliation
  $\cF$ and the relatively compact domain $\cB\subset\cF$ being
  considered.
\end{Rem}

\subsection{Counting algebraic points}

For this section we fix: $\ell\in\N$; a map $\Phi\in\cO(\M)^\ell$
defined over $\K$; an algebraic $\K$-variety $V\subset\M$; and a
$\vxi$-ball $\cB\subset\B_R$ of radius at most $R$. Set
\begin{equation}
  A = A_{V,\Phi,\cB} := \Phi(\cB^2\cap V) \subset \C^\ell.
\end{equation}
Denote
\begin{equation}
  A(g,h) := \{ p\in A : [\Q(p):\Q] \le g \text{ and } h(p)\le h\}.
\end{equation}
Our goal will be to study the sets $A(g,h)$ in the spirit of the
Pila-Wilkie counting theorem \cite{pila-wilkie}. Toward this end we
introduce the following notation.

\begin{Def}\label{def:blocks}
  Let $\cW\subset\C^\ell$ be an irreducible algebraic variety.  We
  denote by $\Sigma(V,\cW;\Phi)$ the union of (i) the points $p$ where
  the germ $\Phi\rest{\cL_p\cap V}$ is not a finite map; (ii) the
  points $p$ where $\Phi(\cL_p\cap V)$ contains one of the analytic
  components of the germ $\cW_{\Phi(p)}$. We omit $\Phi$ from the
  notation if it is clear from the context.
\end{Def}

In most applications $\Phi$ will be a set of coordinates on the leafs
of our foliation and condition (i) will be empty. Condition (ii) then
states that $\Phi(\cL_p\cap V)$ contains a connected semialgebraic set
of positive dimension (namely a component of $\cW$). Our main result
is the following.

\begin{Thm}\label{thm:alg-count}
  Let $\e>0$. There exists a collection of irreducible
  $\Q$-subvarieties $\{\cW_\alpha\subset\C^\ell\}$ such
  that $\dist(\cB,\Sigma(V,\cW_\alpha))<\e$,
  \begin{equation}
    A(g,h) \subset \cup_\alpha \cW_\alpha,
  \end{equation}
  and
  \begin{equation}\label{eq:alg-count-deg}
    \#\{\cW_\alpha\}, \max_\alpha \delta_{\cW_\alpha} =
    \poly(\delta_\vxi,\delta_V,\delta_\Phi,g,h,\log R,\log \e^{-1}).
  \end{equation}
\end{Thm}

As with Theorem~\ref{thm:zero-count}, one can always reduce to the
case $R=1$ in Theorem~\ref{thm:alg-count} by rescaling, and we will
consider only the case $R=1$ in the proof.

\begin{Rem}[Blocks from nearby leafs]
  Theorem~\ref{thm:alg-count} can be viewed as an analog of the
  Pila-Wilkie theorem in its blocks formulation
  \cite{pila:blocks}. Suppose for simplicity that $\Phi$ is such that
  condition $(i)$ in Definition~\ref{def:blocks} is automatically
  satisfied for all leafs. The $\{\cW_\alpha\}$ are similar to blocks
  in the sense that they are algebraic varieties containing all of
  $A(g,h)$. The difference is that in the Pila-Wilkie theorem, these
  blocks are all subsets of $A^\alg$. In Theorem~\ref{thm:alg-count}
  one should think of the set $A$ as belonging to a family $A_\cL$,
  parametrized by varying the leaf $\cL$ while keeping $V,\Phi$
  fixed. The blocks $\cW_\alpha$ correspond to some algebraic part,
  but possibly of an $A_\cL$ for a nearby leaf $\cL$ (at distance $\e$
  from the original leaf). We therefore refer to $\{W_\alpha\}$ as
  blocks coming from nearby leafs.
\end{Rem}

Ideally one would hope to obtain a result
with~\eqref{eq:alg-count-deg} independent of $\e$, which would
eliminate the need to consider blocks from \emph{nearby} leafs and
give a result roughly analogous to a block-counting version of the
Wilkie conjecture. Unfortunately, due to the dependence in our main
theorems on $\log\dist^{-1}(\cB,\Sigma_V)$, one cannot expect to
derive such a result.

On the other hand, in practical applications of the counting theorem
one usually has good control over the possible blocks, not only on
$\cB$ but on all nearby leafs. This occurs because the foliations
normally used in Diophantine applications are highly symmetric,
usually arising as flat structures associated to a principal
$G$-bundle for some algebraic group $G$. This implies that the blocks
from nearby leafs are obtained as symmetric images (by a symmetry
$\e$-close to the identity) of the blocks from leaf $\cL$ itself. In such
cases Theorem~\ref{thm:alg-count} gives an effective polylogarithmic
version of the Pila-Wilkie counting theorem, which usually leads to
refined information for the Diophantine application. We give several
examples of this in~\secref{sec:applications}.

As a simple example of this type we have the following consequence
of Theorem~\ref{thm:alg-count}, in the case the no blocks appear on
any of the leafs.

\begin{Cor}\label{cor:wilkie}
  Suppose that for every $p\in\M$ the germ $\Phi\rest{\cL_p\cap V}$ is
  a finite map, and $\Phi(\cL_p\cap V)$ contains no germs of algebraic
  curves. Then
  \begin{equation}
    \#A(g,h) = \poly_\ell(\delta_\vxi,\delta_V,\delta_\Phi,\log R,g,h).
  \end{equation}
\end{Cor}

\subsection{A result for restricted elementary functions}

Recall the structure of \emph{restricted elementary functions} is
defined by
\begin{equation}
  \R^\RE = (\R,<,+,\cdot,\exp\rest{[0,1]},\sin\rest{[0,\pi]}).
\end{equation}

For a set $A\subset\R^m$ we define the \emph{algebraic part} $A^\alg$
of $A$ to be the union of all connected semialgebraic subsets of $A$
of positive dimension. We define the \emph{transcendental part}
$A^\trans$ of $A$ to be $A\setminus A^\alg$.

In \cite{me:rest-wilkie} together with Novikov we established the
\emph{Wilkie conjecture} for $\R^\RE$-definable sets. Namely,
according to \cite[Theorem~2]{me:rest-wilkie} if $A\subset\R^m$ is
$\R^\RE$-definable then $\#A^\trans(g,h)=\poly_{A,g}(h)$. Replacing
the application of \cite[Proposition~12]{me:rest-wilkie} by the
stronger Proposition~\ref{prop:interpolation} established in the
present paper yields sharp dependence on $g$.

\begin{Thm}\label{thm:RE-wilkie}
  Let $A\subset\R^m$ be $\R^\RE$-definable. Then 
  \begin{equation}
    \#A^\trans(g,h) = \poly_A(g,h).
  \end{equation}
\end{Thm}

We remark that the proof of Proposition~\ref{prop:interpolation} and
consequently Theorem~\ref{thm:RE-wilkie} is self-contained and
independent of the main technical material developed in the present
paper. Still, we thought Theorem~\ref{thm:RE-wilkie} is worth stating
explicitly for its own sake, and for putting
Theorem~\ref{thm:alg-count} into proper context.

\subsection{Comparison with other effective counting results}

For restricted elementary functions, the approach developed in
\cite{me:rest-wilkie} gives results that are strictly stronger than
the results obtained in this paper. This can also be generalized to
holomorphic-Pfaffian functions, including elliptic and abelian
functions. The main limitation of this approach is that it does not
seem to apply to period integrals and other maps that arise in
problems related to variation of Hodge structures. It therefore does
not seem to give an approach to effectivizing the main Diophantine
applications considered in
Theorems~\ref{thm:effective-mz}~and~\ref{thm:ao-decidable}. It does
apply in the context considered in Theorem~\ref{thm:torsion-deg}, but
not in the corresponding analog for Shimura varieties briefly
discussed in~\secref{sec:deg-bounds-future}.

An alternative approach based on the theory of Noetherian functions
has been developed in \cite{me:noetherian-pw}. This class does include
period integrals and related maps. The results of the present paper
have four main advantages:
\begin{enumerate}
\item The asymptotic bounds in Theorem~\ref{thm:alg-count} depend
  polynomially on $g,h$, whereas the results of
  \cite{me:noetherian-pw} are for fixed $g$, and sub-exponential
  $e^{\e h}$ in $h$.
\item The asymptotic bounds in Theorem~\ref{thm:alg-count} depend
  polynomially on the degrees of the equations, whereas in
  \cite{me:noetherian-pw} the dependence is repeated-exponential. The
  sharper dependence allows us to obtain the natural asymptotic
  estimates in the Diophantine applications, leading for instance to
  polynomial-time algorithms.
\item The results of \cite{me:noetherian-pw} deal strictly with
  \emph{semi-Noetherian} sets, i.e. sets defined by means of
  equalities and inequalities but \emph{no
    projections}. Theorem~\ref{thm:alg-count} on the other hand allows
  images under algebraic maps. In many cases, for instance in the
  proof of Theorem~\ref{thm:effective-mz}, the use of projections is
  essential and \cite{me:noetherian-pw} is difficult, if at all
  possible, to use directly.
\item Both the present paper and \cite{me:noetherian-pw} count points
  only in compact domains. However estimates in
  \cite{me:noetherian-pw} grow \emph{polynomially} with the radius $R$
  of a ball containing the domain, whereas in the present paper they
  grow \emph{polylogarithmically}. In many applications this sharper
  asymptotic allows us to deal with non-compact domains by restricting
  to sufficiently large compact subsets.
\end{enumerate}

On the other hand, the approach of \cite{me:noetherian-pw} has one
main advantage: it gives bounds independent of the log-heights of the
equations and the distance to the unlikely intersection
locus. Unfortunately the technical tools used in
\cite{me:noetherian-pw} to achieve this are of a very different nature
and we currently do not see a way to combine these approaches. This
seems to be a fundamental difficulty related to Gabrielov-Khovanskii's
conjecture on effective bounds for systems of Noetherian equations
\cite[Conjectures~1,2]{gk:noetherian}, which is formulated in the
local case and is still open even in this context (though see
\cite{me:deflicity} for a solution under a mild condition).

\subsection{Sketch of the proof}
\label{sec:sketch}

In \cite{me:rest-wilkie} the notion of Weierstrass polydiscs was
introduced for the purpose of studying rational points on analytic
sets. In \cite{me:rest-wilkie} the sets under consideration are
Pfaffian, and an analog of Theorem~\ref{thm:zero-count} (with bounds
depending only on $\deg V$) was already available due to Khovanskii's
theory of Fewnomials \cite{khovanskii:fewnomials}. One of the main
results of \cite{me:rest-wilkie} was a corresponding analog of
Theorem~\ref{thm:wp-cover}, established by combining Khovanskii's
estimates with some ideas related to metric entropy.

In the context of arbitrary foliations there is no known analog for
Khovanskii's theory of Fewnomials. It was therefore reasonable to
expect that the first step toward generalizing the results of
\cite{me:rest-wilkie} would be to establish such a result on counting
intersections, following which one could hopefully deduce a result on
covering by Weierstrass polydiscs using a similar
reduction. Surprisingly, our proof does not follow this line. Instead,
we prove Theorems~\ref{thm:zero-count}~and~\ref{thm:wp-cover} by
simultaneous induction, using crucially the Weierstrass polydisc
construction in dimension $n-1$ when proving the bound on intersection
points in dimension $n$. We briefly review the ideas for the two
simultaneous inductive steps below.

\subsubsection{Poof of Theorem~$\ref{thm:zero-count}_n$ assuming
  Theorem~$\ref{thm:zero-count}_{n-1}$ and Theorem~$\ref{thm:wp-cover}_n$}

We start by reviewing the argument for one-dimensional
foliations. This case is considerably simpler and was essentially
treated in \cite{me:poly-zeros}. The problem in this case reduces to
counting the zeros of a polynomial $P$ restricted to a ball $\cB^2$ in
the trajectory $\gamma$ of a polynomial vector field. Our principal
zero-counting tool is a result from value distribution theory (see
Proposition~\ref{prop:jensen}) stating that
\begin{equation}
  \#\{z\in\cB^2 : P(z)=0\} \le \const\cdot\log \frac{\max_{z\in\cB} |P(z)|}{\max_{z\in\cB^2}|P(z)|}.
\end{equation}
In our context the logarithm of the numerator can be suitably
estimated from above easily, and the key problem is to estimate the
logarithm of the denominator from below.

By the Cauchy estimates, it is enough to give a lower bound
\begin{equation}\label{eq:Pk-estimate}
  \log (1/P^{(k)}(0)) \ge \poly(\delta_\xi,\delta_P,\log\dist^{-1}(0,\Sigma_V))
\end{equation}
for some $k=\poly(\delta_\xi,\delta_P)$. Note that $P^{(k)}=\xi^k P$
are themselves polynomials. Using multiplicity estimates
(e.g. \cite{nesterenko:mult-nonlinear,gabrielov:mult}) one can show
that for $\mu=\poly(\delta_\xi,\delta_P)$, the ideal generated by
these polynomials for $k=1,\ldots,\mu$ defines the variety
$\Sigma_V$. A Diophantine \Lojas. inequality due to Brownawell
\cite{brownawell:null2} then shows that one of these polynomials can
be estimated from below in terms of the distance to $\Sigma_V$
giving~\eqref{eq:Pk-estimate}.

Consider now the higher dimensional setting, where for instance $V$ is
given by $V(P_1,\ldots,P_n)$. The first difficulty in extending the
scheme above to this context is to find a suitable replacement for the
ideal generated by the $\xi$-derivatives. This problem has been
addressed in our joint paper with Novikov \cite{me:mult-ops}, where we
defined a collection of differential operators $\{\mo_\alpha\}$ of
order $k$ on maps $F:\C^n\to\C^n$, such that all operators $\mo(F)$
vanish at a point if and only if that point is a common zeros of
$F_1,\ldots,F_n$ of multiplicity at least $k$. Combined with the
multidimensional multiplicity estimates of Gabrielov-Khovanskii
\cite{gk:noetherian} this allows one to find a multiplicity operator
$\mo(P)$ of absolute value comparable to $\dist(\cB,\Sigma_V)$ (see
Proposition~\ref{cor:mo-v-dist}).

The other, more substantial, difficulty is to find an appropriate
analog for the value distribution theoretic statement. It is well
known that the Nevanlinna-type arguments used above in dimension one
generally become much more complicated to carry out for sets of
codimension greater than one, and indeed this has been the primary
reason that many works on point-counting using value distribution have
been restricted to the one-dimensional case.

Our main new idea is that one can overcome this difficulty by
appealing to the notion of Weierstrass polydiscs. Namely, using the
inductive hypothesis we may reduce to studying the common zeros of
$P_1,\ldots,P_n$ inside a Weierstrass polydisc
$\Delta:=D_z\times\Delta_w$ for the curve
\begin{equation}
  \Gamma:=\cB\cap V(P_1,\ldots,P_{n-1}).
\end{equation}
This is equivalent to studying the zeros of the \emph{analytic
  resultant}
\begin{equation}
  \cR(z) = \prod_{w:(z,w)\in\Gamma\cap\Delta} P_n(z,w).
\end{equation}
We are thus reduced to the case of holomorphic functions of one
variable, and it remains to show that $\cR(z)$ can be estimated from
below in terms of the multiplicity operators (similar to how $P(z)$
was estimated from below in terms of the usual derivatives in the one
dimensional case). This is indeed possible, using some properties of
multiplicity operators developed in \cite{me:mult-ops}, and the
precise technical statement is proved in
Lemma~\ref{lem:mo-v-resultant}.

\subsubsection{Proof of Theorem~$\ref{thm:wp-cover}_n$ assuming
  Theorem~$\ref{thm:zero-count}_{n-1}$}

In \cite{me:rest-wilkie} the proof of the analog of
Theorem~\ref{thm:wp-cover} was based on a simple geometric
observation. Namely, one shows that to construct Weierstrass polydisc
containing a ball of radius $r$ around the origin for a set
$X\subset\cB$ it is essentially enough to find a ball $B'\subset\cB$
of radius $\sim r$ disjoint from $S^1\cdot X$ (where $S^1$ acts on
$\cB$ by scalar multiplication).

To find such a ball, in \cite{me:rest-wilkie} we appeal to Vitushkin's
formula. Unfortunately this real argument would require restricting to
real codimension one sets. Since our inductions works by decreasing
the complex dimension (in order to use arguments from value
distribution theory), this approach is not viable in our
case. Instead, we show in Proposition~\ref{prop:empty-ball} that one
can always find a ball $B'$ as above with
\begin{equation}\label{eq:intro-wp-v-vol}
  1/r = O(\sqrt[\alpha]{\vol(X)}), \qquad \alpha:=2n-2m-1.
\end{equation}
The proof is based on the fact that the volume of a complex analytic
set passing through the origin of a ball of radius $\e$ is at least
$\const\cdot\e^{2\dim X}$. An analytic set that meets many disjoint
balls must therefore have large volume. We remark that this is an
essentially complex-geometric statement which fails in the real
setting.

Having established the estimate~\eqref{eq:intro-wp-v-vol}, we see that
to construct a reasonably large Weierstrass polydisc around the origin
for $\cB\cap V$ (and then cover $\cB^2$ by a simple subdivision
argument) it is enough to estimate the volume of this set. Moreover, a
simple integral estimate shows that having found such a Weierstrass
polydisc $\Delta$, the multiplicity $e(X,\Delta)$ is also upper
bounded in terms of $\vol(\cB\cap V)$. We reduce the estimation of
this volume, using a complex analytic version of Crofton's formula, to
counting the intersections of $\cB\cap V$ with all linear planes of
complementary dimension. We realize these planes as leafs of a new
(lower-dimensional) foliated space and finish the proof by inductive
application of Theorem~\ref{thm:zero-count}.

\subsubsection{Under the rug}

The two inductive steps of our proof are carried out by restricting
our foliation $\cF$ to its linear sub-foliations (where the leafs are
given by linear subspaces, in the $\vxi$-variables, of the original
leafs). It may happen coincidentally that new unlikely intersections
are created in this process. For example, if $P_1,P_2$ are two
polynomial equations intersecting properly with a two-dimensional leaf
$\cL_p$, it may happen that the restriction of $P_1$ to some
one-dimensional $\vxi$-linear subspace of $\cL_p$ vanishes
identically. In this case one cannot control the
$\log\dist^{-1}(\cB,\Sigma_V)$ term coming up in the induction.

To avoid this problem, we note that the particular choice of linear
$\vxi$-coordinates plays no special role in the argument, and one can
use any other parametrization (sufficiently close to the identity to
maintain control over the distortion of the $\vxi$-unit balls). We
therefore replace the vector fields $\vxi$ by a new tuple $\tilde\vxi$
generating the same foliation $\cF$, but producing a different
parametrization of the leafs. We show that for a sufficiently generic
choice of $\tilde\vxi$ one can avoid creating new unlikely
intersections in any of the linear sections considered in the
proof. The main technical difficulty is to show that $\tilde\vxi$ can
be constructed with $\delta_\vxi=\poly(\delta_\vxi,\delta_V)$.

\subsubsection{Counting algebraic points}

Having proved the general results on counting intersection points
between algebraic varieties and leafs and covering such intersections
with a bounded number of Weierstrass polydiscs, one can attempt to
approach a Pila-Wilkie type counting theorem using the strategy
employed in \cite{me:analytic-interpolation,me:rest-wilkie}. A direct
application of this strategy yields adequate estimates for the
algebraic points in a fixed number field (as a function of height),
but fails to produce such estimates when one fixes only the degree of
the number field. To achieve this greater generality we use an
alternative approach suggested by Wilkie in \cite{wilkie:pw-notes},
which replaces the interpolation determinant method by a use of the
Thue-Siegel lemma. We remark that Habegger has used this approach in
his work on an approximate Pila-Wilkie type theorem
\cite{habegger:approx}, and our result is influenced by his idea.

Since we, unlike Wilkie and Habegger, use Weierstrass polydiscs in
place of the traditional $C^r$-smooth parametrization some technical
preparations parallel to \cite{wilkie:pw-notes,habegger:approx} must
be made. This material is developed in~\secref{sec:interpolation}.

\section{Multiplicity operators and local geometry on $\cF$}
\label{sec:mult-ops}

Let $F=(F_1,\ldots,F_n)$ denote an $n$-tuple of holomorphic functions
in some domain $\Omega\subset\C^n$. The paper \cite{me:mult-ops}
defines a collection $\{\bmo_B^\alpha\}$ of ``basic multiplicity
operators'' of order $k$. These are partial differential operators of
order $k$, i.e. polynomial combinations of $F_1,\ldots,F_n$ and their
first $k$ derivatives\footnote{We remark that in \cite{me:mult-ops} a
  general multiplicity operator is defined as an element of the convex
  hull of the basic ones; however in this paper, since we are
  concerned with heights over a number field, we will stick to using
  only the basic operators and write ``multiplicity operator'' for a
  basic operator.}. We will usually denote a multiplicity operator of
order $k$ by $\mo$ and write $\mo_p(F)$ for $[\mo(F)](p)$.

The key defining property of the multiplicity operators is the
following. Denote by $\mult_p F$ the multiplicity of $p$ as a common
zero of $F_1,\ldots,F_n$ (with $\mult_pF=0$ if $p$ is not a common
zero and $\mult_p F=0$ if $p$ is a non-isolated zero).
\begin{Prop}[\protect{\cite[Proposition~5]{me:mult-ops}}]\label{prop:mo-basic}
  We have $\mult_p F>k$ if and only if $\mo_pF=0$ for all multiplicity
  operators of order $k$.
\end{Prop}

\subsection{Multiplicity operators and Weierstrass polydiscs}

In this section we denote by $B\subset\C^n$ the unit ball. The norm
$\norm{\cdot}$ always denotes the maximum norm. We will need the
following basic lemma on multiplicity operators.

\begin{Lem}\label{lem:mo-lower-bound}
  Let $F_1,\ldots,F_n:B\to D(1)$. Suppose that $s=\abs{\mo_0F}\neq0$
  for some multiplicity operator $\mo$. Let $\ell\in(\C^n)^*$ have
  unit norm and let $0<\rho<s$. Then there is a ball $B'$ around the
  origin of radius at least $s/\poly_n(k)$ and a union of at most $k$
  discs $U_\rho$ of total radius at most $\poly_n(k)\cdot\rho$ such
  that
  \begin{equation}
    \vz\in B'\setminus\ell^{-1}(U_\rho) \implies \log \norm{F(\vz)} \ge (k+1)\log\rho-\poly_n(k).
  \end{equation}
\end{Lem}
\begin{proof}
  The statement follows from the proof of
  \cite[Theorem~2]{me:mult-ops}. To see this it suffices to check in
  the proof that the various constants appearing there indeed have
  logarithms of order $\poly_n(k)$. This boils down to estimating the
  constants $C_k$ and $\CD_{n,k}$. The former is given explicitly in
  \cite[Lemma~4.1]{me:polyfuchs} in the form $C_k=2^{-O(k)}$. The
  latter arises in the proof of \cite[Proposition~6]{me:mult-ops} from
  applying Cramer's rule to a determinant of size $\poly_n(k)$, and is
  easily seen to satisfy $\log\CD_{n,k}=\poly_n(k)$.
\end{proof}

We now state a result relating the multiplicity operators to the
construction of a Weierstrass polydisc for a curve.

\begin{Lem}\label{lem:mo-v-wp}
  Let $F_1,\ldots,F_{n-1}:B\to D(1)$. Suppose that
  $s=\abs{\mo_0F}\neq0$ for some $(n-1)$-dimensional multiplicity
  operator $\mo$ with respect to the variables
  $\vw=z_2,\ldots,z_n$. Then there exists a Weierstrass polydisc in
  the standard coordinates $\Delta=D(r_1)\times\cdots\times D(r_n)$
  with all the radii satisfying
  \begin{equation}
    \log r_i \ge \poly_n(k) \log s.
  \end{equation}
\end{Lem}
\begin{proof}
  We claim that one can find a polydisc
  $\Delta_w=D(r_2)\times\cdots\times D(r_n)$ such that
  \begin{equation}
    \log \norm{F(0,\vw)}\ge (k+1)\log s-\poly_n(k) \qquad \text{for every }\vw\in\partial\Delta_w
  \end{equation}
  and moreover
  \begin{equation}
    \log r_i \ge \poly_n(k)\log s \text{ for } i=2,\ldots,n.
  \end{equation}
  To prove this apply Lemma~\ref{lem:mo-lower-bound} to $F(0,\vw)$
  with $\ell$ given by each of the $\vz_2,\ldots,\vz_n$-coordinates
  with a suitable choice $\rho=s/\poly_n(k)$, and then choose
  $\Delta_w$ to be a polydisc inside the balls $B'$ and with each
  $\partial D(r_j)$ disjoint from the set $U_\rho$ obtained for
  $\ell=\vz_j$.

  Since $F_1,\ldots,F_{n-1}$ have unit maximum norms, their
  derivatives are bounded by $O(1)$ in $B^2$ by the Cauchy
  estimate. It follows that $F(z,\vw)$ cannot vanish on
  $\partial\Delta_w$ for $z\in D(r_1)$ where
  \begin{equation}
    \log r_1\sim (k+1)\log s-\poly_n(k)
  \end{equation}
  so $D(r_1)\times\Delta_w$ indeed gives a Weierstrass polydisc
  satisfying the final condition $\log r_1\ge\poly_n(k) \log s$.
\end{proof}

Suppose that $\Gamma\subset\C^n$ is an analytic curve,
$\Delta=D_z\times\Delta_w$ is a Weierstrass polydisc for $\Gamma$ and
$G:\Delta\to\C$ is holomorphic.

\begin{Def}
  We define the \emph{analytic resultant} of $G$ with respect to
  $\Delta$ to be the holomorphic function
  $\cR_{\Delta,\Gamma}(G):D_z\to\C$ given by
  \begin{equation}\label{eq:analytic-res}
    \cR_{\Delta,\Gamma}(G) = \prod_{w:(z,w)\in\Gamma\cap\Delta} G(z,w).
  \end{equation}
\end{Def}

Our second result concerns a lower estimate for analytic resultants in
terms of multiplicity operators.

\begin{Lem}\label{lem:mo-v-resultant}
  Let $F_1,\ldots,F_n:B\to D(1)$ be holomorphic. Set
  $\Gamma=\{F_1=\cdots=F_{n-1}=0\}$ and suppose that
  $\Delta=D(r)\times\Delta_w\subset B$ is a Weierstrass polydisc in
  the standard coordinates for $\Gamma$ with multiplicity
  $\mu$. Suppose that $s=\abs{\mo_0(F)}\neq0$ for some multiplicity
  operator $\mo$. Let $0<\rho<s$. Then for $z$ in a ball of radius
  $\Omega_n(s)$ around the origin and outside a union of balls of
  radius $O_n(\rho)$ we have
  \begin{equation}
    \log |R(z)| > \mu\cdot((k+1)\log\rho-\poly_n(k)), \qquad R:=\cR_{\Delta,\Gamma}(G):D(r)\to\C.
  \end{equation}
\end{Lem}
\begin{proof}
  Apply Lemma~\ref{lem:mo-lower-bound} with $\ell=\vz_1$ and
  $\rho$. We see that $\log \norm{F(\vz)}\ge(k+1)\log\rho-\poly_n(k)$ in
  a ball $B'$ of radius $\Omega_n(s)$ whenever $\vz_1$ lies outside
  $U_\rho$. In particular this is true for the $\mu$ points over
  $\vz_1$ where $F_1,\ldots,F_{n-1}$ vanish, and at these points we
  obtain the same estimate for $\log|F_n(\vz)|$. Taking product over
  the $\mu$ different points proves the statement.
\end{proof}

\subsection{Multiplicity operators along $\cF$}

When $P=(P_1,\ldots,P_n)\in\cO(\M)^n$ we may apply the multiplicity
operator $\mo$ to $P$ by evaluating the derivatives along
$\vxi_1,\ldots,\vxi_n$. This amount to computing, for each point
$p\in\M$, the multiplicity operator of $P\rest\cL_p$ in the
$\vxi$-chart. 

\begin{Lem}\label{lem:mo-complexity}
  For any multiplicity operator $\mo$ we have
  \begin{equation}
    \delta(\mo P) = \poly(\delta_P,\delta_\vxi,k).
  \end{equation}
\end{Lem}
\begin{proof}
  This is a simple computation owing to the fact that $\mo$ is defined
  by expanding a determinant of size $\poly_n(k)$ with entries defined
  in terms of $P$ and its $\vxi$-derivatives up to order $k$.
\end{proof}

We will require the following result of Gabrielov-Khovanskii
\cite{gk:noetherian}.

\begin{Thm}\label{thm:gk-mult}
  With $P$ as above and $p\in\M\setminus\Sigma_{V(P)}$,
  \begin{equation}
    \mult_p P < \poly(\deg\vxi, \deg P).
  \end{equation}
\end{Thm}

As a consequence we have the following.

\begin{Prop}\label{prop:SigmaV-complexity}
  Let $V\subset\M$ be a complete intersection $V=V(P_1,\ldots,P_m)$
  with $m\le n$. Then
  \begin{equation}
    \delta(\Sigma_V) = \poly(\delta_\vxi,\delta_V).
  \end{equation}
  Moreover if $m=n$ then $\Sigma_V$ is set-theoretically cut out by
  the functions $\{\mo(P)\}$ where $\mo$ varies over all multiplicity
  operators of order $k=\poly(\deg\vxi,\deg P)$.
\end{Prop}
\begin{proof}
  We have $p\in\Sigma_V$ if and only if $p\in\Sigma_\cF$ or
  $\dim(\cL_p\cap V)>n-m$. Since clearly
  $\delta(\Sigma_\cF)=\poly(\delta_\vxi)$ we only have to write
  equations for the latter condition. This is equivalent to the
  statement that for every $\vxi$-linear subspace of $\cL_p$ of
  dimension $m$ the intersection $V\cap L$ is
  non-isolated, i.e. has infinite multiplicity. We express this using
  multiplicity operators as follows.

  Let $\vc^1,\ldots,\vc^m$ be $n$-tuples of indeterminate coefficients
  and let
  \begin{equation}
    \vxi_\vc=(\vc^1\cdot\vxi,\ldots,\vc^m\cdot\vxi)
  \end{equation}
  denote the sub-foliation of $\vxi$ generated by the corresponding
  linear combinations. Then for every $p\in\M\setminus\Sigma_\cF$ we
  obtain a linear subspace $\cL_{p,\vc}\subset\cL_p$ and we seek to
  express the condition that $\cL_{p,\vc}\cap V$ is an intersection of
  infinite multiplicity for every $\vc$. By Theorem~\ref{thm:gk-mult},
  if the intersection multiplicity is finite then it is bounded by
  $k=\poly(\deg\vxi,\deg P)$. It is enough to express the condition
  that the multiplicity exceeds this number for every $\vc$.
  According to Proposition~\ref{prop:mo-basic}, for every fixed value
  of $\vc$ this condition can be expressed by considering all
  multiplicity operators $\mo(P)$ with respect to $\vxi_\vc$. Expanding
  these expressions with respect to the variables $\vc$ and taking the
  ideal generated by all the coefficients we obtain equations for the
  vanishing for every $\vc$. The estimates on the degrees and heights
  of these equations follow easily from Lemma~\ref{lem:mo-complexity}.
\end{proof}

We record a useful corollary of
Proposition~\ref{prop:SigmaV-complexity}.

\begin{Cor}\label{cor:mo-v-dist}
  Let $V=V(P_1,\ldots,P_n)$ be a complete intersection and
  $p\in\B$. There exists a multiplicity operator $\mo$ of order
  $k=\poly(\deg\vxi,\deg V)$ such that
  \begin{equation}
    \log|\mo_p(P)| \ge \poly(\delta_\vxi,\delta_V)\cdot \log\dist(p,\Sigma_V).
  \end{equation}
\end{Cor}
\begin{proof}
  According to Proposition~\ref{prop:SigmaV-complexity} the set
  $\Sigma_V$ is set-theoretically cut out by the multiplicity
  operators $\mo(P)$ as above. Since the degrees and heights of these
  polynomials are bounded according to
  Proposition~\ref{lem:mo-complexity}, the result follows by
  application of the Diophantine \Lojas. inequality due to Brownawell
  \cite{brownawell:null2}.
\end{proof}

\section{Covering by Weierstrass polydiscs}

Let $B\subset\C^n$ denote the unit ball around the origin and
$X\subset B$ an analytic subset of pure dimension $m$. In this section
we prove that one can find a Weierstrass polydisc around the origin
for $X$, where the size of the polydisc depends on the volume of $X$.

For a subset $A\subset\C^n$ denote by $N(A,\e)$ the size the smallest
$\e$-net in $A$, and by $S(A,\e)$ the size of the maximal
$\e$-separated set in $A$. One easily checks that
\begin{equation}
  S(A,2\e) \le N(A,\e) \le S(A,\e).
\end{equation}

\begin{Lem}\label{lem:vol-v-net}
  For $\e\le1$ we have
  \begin{equation}
    S(X\cap B^2,\e) \le \frac{2^m}{c(m)} \vol(X)\cdot \e^{-2m}
  \end{equation}
  where $c(m)$ denotes the volume of the unit ball in $\C^m$.
\end{Lem}
\begin{proof}
  Suppose $S\subset X\cap B^2$ is an $\e$-separated set. Then balls
  $B_p:=B(p,\e/2)$ for $p\in S$ are disjoint, and according to
  \cite[Theorem~15.3]{chirka} we have
  \begin{equation}
    \vol(X\cap B_p) \ge c(n)(\e/2)^{2m}.
  \end{equation}
  The conclusion follows since the disjoint union of these sets is
  contained in $X$.
\end{proof}

Let the unit circle $S^1\subset\C$ act on $\C^n$ by scalar
multiplication.
\begin{Lem}\label{lem:S1-net}
  Let $A\subset B$. Then
  \begin{equation}
    N(S^1\cdot A,2\e) \le (1+\lfloor \pi/\e \rfloor) \cdot N(A,\e)
  \end{equation}
\end{Lem}
\begin{proof}
  Build a $2\e$-net for $S^1\cdot A$ by multiplying an $\e$-net in
  $S^1$ by an $\e$-net in $A$.
\end{proof}

The following proposition is our key technical result.
\begin{Prop}\label{prop:empty-ball}
  There exists a ball $B'\subset B$ of radius $\e$ disjoint from
  $S^1\cdot X$, where
  \begin{equation}
    1/\e = O_n(\sqrt[\alpha]{\vol(X)}), \qquad \alpha:=2n-2m-1.
  \end{equation}
\end{Prop}
\begin{proof}
  Set $X'=S^1\cdot(X\cap B^2)$. By Lemmas~\ref{lem:vol-v-net}
  and~\ref{lem:S1-net} we have
  \begin{equation}
    N(X',\e) = O_n(\vol(X)\e^{-2m-1}).    
  \end{equation}
  On the other hand clearly
  \begin{equation}
    S(B^2,\e) = \Theta_n(\e^{-2n}).
  \end{equation}
  Suppose that $N$ is an $\e$-net for $X'$ and $S$ is a
  $4\e$-separated set in $B^2$. Suppose that every $\e$-ball $B_p$
  around a point $p\in S$ meets $X'$. Then the $B_p^{1/2}$ meets
  $N$. Since $S$ is $4\e$-separated no two balls $B_p^{1/2},B_q^{1/2}$
  for $p,q\in S$ meet the same point of $N$, so $\#S\le\#N$. In
  conclusion, as soon as we have $S(B^2,\e)>N(X',\e)$ there exists an
  $\e$-ball $B_p$ that does not meet $X'$.
\end{proof}

As a corollary we obtain our main result for this section.

\begin{Cor}\label{cor:wp-v-vol}
  There exists a Weierstrass polydisc $\Delta\subset B$ for $X$ which
  contains $B^\eta$, where $\eta=\poly_n(\vol(X))$. Moreover $e(X,\Delta)=\poly_n(\vol(X))$.
\end{Cor}
\begin{proof}
  The proof of the first part is the same as
  \cite[Theorem~7]{me:rest-wilkie}, where we replace the use of
  Vithushkin's formula and sub-Pfaffian arguments by
  Proposition~\ref{prop:empty-ball}. Briefly, after finding a ball
  $B'$ disjoint from $S^1\cdot X$ one notes that $B'$ contains a set
  which has the form $\Delta\times\partial D(r)$ in some unitary
  coordinates system, where the radii of $\Delta$ and $D(r)$ are
  roughly the same as the radius of $B'$. It is then easy to reduce
  the problem to finding a Weierstrass polydisc for $\pi(X)$ inside
  $\Delta$. Since $\pi(X)$ is again an analytic set and
  $\vol(\pi(X))\le\vol(X)$ the proof is concluded by induction over
  the dimension.

  For the second part, write
  \begin{multline}
    \vol(X\cap\Delta) = \int_{X\cap\Delta} \d\vol_X \ge \int_{X\cap\Delta}(\pi^*\d\vol_{\Delta_z}) \\
    = e(X,\Delta) \int_{\Delta_z} \d\vol_{\Delta_z} = e(X,\Delta) \vol(\Delta_z)
  \end{multline}
  and note that $\vol(\Delta_z)^{-1}=\poly_n(\vol(X))$ by what was
  already proved.
\end{proof}

\section{Achieving general position}
\label{sec:general-position}

Let $V\subset\M$ be a variety of pure dimension $m$. We will assume
until~\secref{sec:CI-choose} that $V$ is a complete intersection
variety defined by
$Q_1,\ldots,Q_{n-m}\in\cO(\M)$. In~\secref{sec:CI-choose} we prove a
result that allows to reduce the general case to the case of complete
intersections.

As explained in~\secref{sec:sketch} a part of our inductive scheme
involves studying intersections between the variety defined by
$Q_1,\ldots,Q_k$ and sub-foliations of $\cF$ defined by $k$-dimensional
linear subspaces of $\<\xi_1,\ldots,\xi_n\>$. To carry this out
uniformly we add the coefficients of such a linear combination to
$\M$. It may happen that the process of restricting to a linear
sub-foliation introduces new unlikely intersections (e.g. if $Q_1$,
while not vanishing identically on a leaf, happens to vanish on a
linear hyperplane in the $\vxi$-coordinates). To avoid such
degeneracies we perturb the time parametrization, changing the fields
$\vxi$ while preserving the leafs $\cL_p$ themselves. We show that
this can be done while preserving suitable control over $\delta_\vxi$.

\subsection{Parametrizing linear sub-foliations}

Let $k\le n$ and let $A(n,k)$ denote the affine variety of full rank
matrices $(\valpha_1,\ldots,\valpha_k)\in\Mat_{n\times k}$. Let
$L_k\M:=A(n,k)\times M$ and consider the vector fields
\begin{equation}
  L_k(\vxi)_i = \valpha_i\cdot\vxi \qquad i=1,\ldots,k.
\end{equation}
The leafs of $L_k\M$ with $L_k(\vxi)$ correspond to the leafs obtained
by choosing a $k$-dimensional subspace of $\<\vxi_1,\ldots,\vxi_n\>$
and using it to span a $k$-dimensional sub-foliation of $\cF$.

\subsection{Main statement}
\label{sec:perturb-main}

Our goal is to construct an affine variety $\tilde\M:=N\times\M$
depending only on $\M$, and vector fields $\tilde\vxi$ depending on
$\M,V$ with the following properties.

\begin{enumerate}
\item If we denote by $\pi_\M:\tilde \M\to \M$ the projection and by
  $\phi_p,\tilde\phi_{a,p}$ the $\vxi,\tilde\vxi$ charts respectively,
  then for any $(a,p)\in\tilde\M$ we have
  $\pi_\M\circ\tilde\phi_{a,p}=\phi_p\circ \Phi_{a,p}$ where
  $\Phi_{a,p}$ is the germ of a self map of $(\C^n,0)$. In particular
  $\cL_p=\pi_\M(\cL_{a,p})$.
\item Whenever $\phi_p$ extends to the unit ball, the germ
  $\Phi_{a,p}$ extends to $B^2 $ and
  \begin{equation}
    \norm{\Phi_{a,p}-\id}_{B^2}<0.1.
  \end{equation}
  In other words the reparametrization is close to the identity.
\item We have effective estimates
  \begin{align}
    \deg\tilde\vxi&=\deg\vxi+O(1) & h(\tilde\vxi)&=\poly(\delta_\vxi,\deg V).
  \end{align}
\item For $k=\codim V$, if we set $\tilde\M_k:=L_k\tilde\M$ and denote
  by $\tilde V$ the natural pullback to $\tilde\M_k$ then
  \begin{equation}
    \pi_\M(\Sigma_{\tilde V})\subset \Sigma_V.
  \end{equation}
  In other words no ``new'' unlikely intersections are formed when
  considering linear sub-foliations of $\tilde\M,\tilde\vxi$.
\end{enumerate}

We also remark that one can similarly achieve general position with
respect to any $O(1)$ different varieties $V_i\subset\M$ by the same
argument.

\subsection{Polynomial time reparametrization}
\label{sec:reparametrization}


Fix $D\in\N$ and let $\cM_D$ denote the space of polynomial maps
$\Phi:\C^n\to\C^n$ with coordinatewise degree at most $D$. Let $P_D(\M)$
denote the affine variety obtained from $\cM_D\times\C^n_s\times\M$ by
imposing the condition $\det\pd{\Phi(\vs)}{\vs}\neq0$ where we use $\Phi$
for the coordinate on $\cM_D$ and $\vs$ for the coordinate on
$\C^n$. Consider the vector fields
\begin{equation}
  P_D(\xi_i) = \pd{}{\vs_i} + \pd{\Phi(\vs)}{\vs}\cdot\vxi.
\end{equation}
Then the local $P_D(\vxi)$-chart at a point $(\Phi,a,p)$ is given
by
\begin{equation}
  \phi_{\Phi,a,p}(\vx) = (\Phi,a+\vx,\phi_p(\Phi(a+\vx)-\Phi(a))).
\end{equation}
In particular, the projection of the leaf $P_D(\cL)_{\Phi,a,p}$ to
$\M$ is the germ $\cL_p$, but the time parametrization is adjusted
according to $\Phi$ around $a$.

\subsection{Codimension of unlikely intersection}
\label{sec:reparam-codim}
  
Set
\begin{equation}
  \tilde \M = L_k(P_D\M)= A(n,k)\times\cM_D\times\C^n_s\times\M
\end{equation}
Denote by $\tilde V$ the pullback of $V$ to $\tilde\M$.

\begin{Lem}\label{lem:unlikely-codim}
  Let $p\in\M\setminus\Sigma_V$, $A\in A(n,k)$ and
  $a\in\C^n_s$. Then the set
  \begin{equation}
    \{\Phi\in\cM_D : (A,\Phi,a,p)\in\Sigma_{\tilde V} \}\subset\cM_D
  \end{equation}
  is algebraic of codimension at least $D$.
\end{Lem}
\begin{proof}
  Algebraicitiy follows from Proposition~\ref{prop:SigmaV-complexity}.
  Replacing $a$ by $0$ and $\Phi(\vx)$ by $\Phi(a+\vx)-\Phi(a)$ we may
  assume without loss of generality that $a=0$. Similarly replacing
  $A$ by $(\vxi_1,\ldots,\vxi_k)$ and $\Phi(\vx)$ by its appropriate
  linear change of variable we may assume without loss of generality
  that $A=(\vxi_1,\ldots,\vxi_k)$.

  Denote $\Phi'_j=\Phi_j-\Phi_j(0)$. Then the leaf at
  $(A,\Phi,0,p)$ is defined by
  \begin{equation}
    (A,\Phi,0)\times\cL_p', \qquad \cL_p'=\cL_p\cap\{\Phi'_{k+1}=\ldots=\Phi'_n=0\}.
  \end{equation}
  We must check when the intersection of $\cL_p'$ and $V$ is a
  complete intersection. It is enough to bound the codimension of the
  condition that $\Phi'_{k+1}$ vanishes identically on (a component
  of) $\cL_p\cap V$, that $\Phi'_{k+2}$ vanishes identically on (a
  component of) $\cL_p\cap\{\Phi'_{k+1}=0\}\cap V$, and so on.

  From the above we conclude that it is enough to prove the following
  simple claim: let $\gamma\subset(\C^n,0)$ be the germ of an analytic
  curve. Then the set of polynomials of degree at most $D$ without a
  free term vanishing identically on $\gamma$ has codimension at least
  $D$. Note that this set is linear. Choose $t$ to be a (linear)
  coordinate on $\C^n$ which is non-constant on $\gamma$. Then clearly
  $t,\dots,t^D$ are linearly independent on $\gamma$ and the claim
  follows.
\end{proof}

Now choose $D=\dim A(n,k)+n+\dim\M+1$. Denote by
$\pi_\Phi:\tilde\M\to\cM_D$ the projection. Then by a dimension
counting argument using Lemma~\ref{lem:unlikely-codim} the codimension
of $\pi_\Phi(\Sigma_{\tilde V})$ is positive. By
Proposition~\ref{prop:SigmaV-complexity} the degree of the Zariski
closure $Z:=\clo\pi_\Phi(\Sigma_{\tilde V})$ is bounded by
$\poly(\deg V,\deg\vxi)$. If we choose any $\Phi_0\not\in Z$ and
restrict $\tilde\M$ to $\Phi=\Phi_0$ then the final condition
in~\secref{sec:perturb-main} is satisfied by definition. It remains
only to show that $\Phi_0$ can be chosen close to the identity map and
with appropriately bounded height. This follows immediately from the
following general statement.

\begin{Lem}\label{lem:Q-pt-select}
  Let $Z\subset\A^N$ be an affine subvariety of total degree at most
  $d$. Then there exists a point $\vx\in\Q^N\setminus Z$ satisfying
  $\norm{\vx}_\infty\le1$ and $H(\vx)\le d$.
\end{Lem}
\begin{proof}
  Let $C\subset\C$ denote the set of points $z$ such that $Z$ has a
  component contained in $\{\vx_1=z\}$. Clearly $\#C<d$. Choose
  $x\in [-1,1]\cap(\Q\setminus C)$ with $H(x)\le d$. The claim now follows by
  induction over $N$ for the variety $Z\cap\{\vx_1=x\}$, naturally
  identified as a subvariety of $\A^{N-1}$.
\end{proof}

\subsection{Generic choice of a complete intersection}
\label{sec:CI-choose}

Let $V\subset\M$ be a variety defined over $\K$. In this section we
show that one can choose a complete intersection $W$ containing $V$
with $\Sigma_W$ being ``as small as possible'' and with effective
control over $\delta_W$. We'll need the following elementary lemma.

\begin{Lem}
  The variety $V$ is set-theoretically cut out by a collection of
  polynomial equations $P_1,\ldots,P_S$ with
  $\delta(P_\alpha)=\poly(\delta_V)$ and $S$ depending only on the
  dimension of the ambient space of $\M$.
\end{Lem}
\begin{proof}
  Recall that we define $h(V)$ in terms of the height of its Chow
  coordinates. The statement thus follows from a classical
  construction due to Chow and van der Waerden that produces a
  canonical system of equations for $V$ in terms of the Chow
  coordinates \cite[Corollary~3.2.6]{gkz}.
\end{proof}

The following is our main result for this section.

\begin{Prop}\label{prop:CI-choose}
  Let $0\le m\le\dim\M$ be an integer. There exists a
  complete-intersection $W$ of pure codimension $m$ that contains $V$
  and satisfies $\delta_W=\poly(\delta_\vxi,\delta_V)$ and
  \begin{equation}\label{eq:CI-choose}
    \Sigma_W = \{p\in\M : \dim(V\cap\cL_p)>n-m\}.
  \end{equation}
\end{Prop}
\begin{proof}
  We remark that the inclusion $\subset$ in~\eqref{eq:CI-choose} is
  trivial. Suppose that we have already constructed a
  complete-intersection $W_k$ of pure codimension $k<m$ satisfying the
  conditions. We will show how to choose a polynomial equation $P$
  vanishing on $V$, with $\delta_P$ bounded, and such that
  $W_{k+1}=W\cap V(P)$ satisfies
  \begin{equation}
    \Sigma_{W_{k+1}} = \{p\in\M : \dim(V\cap\cL_p)>n-k-1\}.
  \end{equation}
  The claim then follows by induction on $k$.
  
  Let $D=\dim\M+1$ and let $\cP_D$ denote the space of polynomials in
  the ambient space of $\M$ of degree at most $D$. Consider
  $\tilde\M:=\cP_D^S\times\M$ and let $\tilde P\in\cO(\tilde\M)$ be
  given by
  \begin{equation}
    \tilde P(Q_1,\ldots,Q_S,\cdot) = Q_1P_1+\cdots+Q_SP_S, \qquad (Q_1,\ldots,Q_S)\in\cP_D^S.
  \end{equation}
  Set $\tilde W_{k+1}=\tilde W_k\cap V(\tilde P)$ where
  $\tilde W_k:=\cP_D^S\times W_k$.

  Let $p$ satisfy $\dim(V\cap\cL_p)<n-k$. By assumption
  $p\not\in\Sigma_{W_k}$. We claim that the codimension in $\cP_D^S$
  of the set
  \begin{equation}
    \{ Q\in\cP_D^S : (Q,p)\in\Sigma_{\tilde W_{k+1}}\}
  \end{equation}
  is at least $D$. Indeed, the condition is equivalent to the fact
  that $\tilde P$ does not vanish identically on any of the
  irreducible components of $\cL_p\cap W_k$. It is enough to check the
  codimension for each component $C$ separately. Since $V$ is
  set-theoretically cut out by $P_1,\ldots,P_S$ and
  $\dim(V\cap\cL_p)<n-k$, one of the polynomials $P_j$, say without
  loss of generality $P_1$, does not vanish identically on $C$. Then
  for any fixed value of $Q_2,\ldots,Q_S$, at most one value of
  $Q_1\rest C$ can give $\tilde P\rest C\equiv0$, and we have already
  seen in the proof of Lemma~\ref{lem:unlikely-codim} that the
  codimension of this affine linear condition is at least $D$.

  We now finish as in~\secref{sec:reparam-codim}. Namely, by
  Proposition~\ref{prop:SigmaV-complexity} we see that
  $\Sigma_{\tilde W_{k+1}}$ is algebraic and
  \begin{equation}
    \deg \Sigma_{\tilde W_{k+1}}=\poly(\deg(\vxi),\deg(V)).
  \end{equation}
  Set $Z=\clo \pi(\Sigma_{\tilde W_{k+1}})$ where
  $\pi:\tilde M\to\cP_D^S$ and note that by a dimension counting
  argument $Z$ has positive codimension. Choosing a point $Q\not\in Z$
  using Lemma~\ref{lem:Q-pt-select} and setting $P=\tilde P(Q,\cdot)$
  finishes the proof.
\end{proof}

\section{Proofs of the main theorems}

In this section we prove Theorem~\ref{thm:zero-count} and
Theorem~\ref{thm:wp-cover} by a simultaneous induction.  We will
assume in both proofs that $V$ is given by a complete intersection
$V=V(P_1,\ldots,P_m)$. For the general case we replace $V$ by a
complete intersection $W$ containing it as in
Proposition~\ref{prop:CI-choose}. Since $\Sigma_V=\Sigma_W$, the
statements for $V$ follow immediately from the statements for $W$.

To avoid repeating the expression
$\poly(\delta_\vxi,\delta_V,\log\dist^{-1}(\cB,\Sigma_V))$ we will say
simply that a quantity is \emph{appropriately bounded} if it admits
such a bound. Recall that, as explained in~\secref{sec:main-results},
we can and do assume that $R=1$ below.

\subsection{Proof of Theorem~\ref{thm:zero-count}}

We prove Theorem~\ref{thm:zero-count} in dimension $n$ assuming that
Theorem~\ref{thm:zero-count} holds for dimension at most $n-1$ and
that Theorem~\ref{thm:wp-cover} holds for dimension at most $n$.

Let $V':=V(P_1,\ldots,P_{n-1})$. Note that
$\Sigma_{V'}\subset\Sigma_V$. We start by passing to general position
with respect to $V$ and $V'$ as in~\secref{sec:perturb-main}. This has
the effect of slightly reparametrizing the time variables, and in the
new parametrization the original balls $\cB,\cB^2$ are contained in
balls of radius slightly larger than $1,1/2$. However dividing these
balls into $O(1)$ balls and rescaling time (i.e. rescaling $\vxi$), we
see that it is enough to prove Theorem~\ref{thm:zero-count} for
$\cB,\cB^2$ in the new parametrization.

Applying Theorem~\ref{thm:wp-cover} we
construct a collection of Weierstrass polydiscs
$\{\Delta_\alpha\subset\cB\}$ for $V'$ such that the union of the
$\Delta^2_\alpha$ covers $\cB^2$. Since $\#\{\Delta_\alpha\}$ is
appropriately bounded it will suffice to count the zeros of $P_n$ on
$V'$ inside each $\Delta^2_\alpha$ separately. Fix one such polydisc
$\Delta:=\Delta_\alpha$ and set $\Delta=D_z\times\Delta_w$ (in some
unitary system of coordinates). We also have the
$\mu:=e(V'\cap\cB,\Delta)$ is appropriately bounded.

Recall the analytic resultant defined in~\eqref{eq:analytic-res}. The
zeros of $P_n$ on $V'$ inside $\Delta$ correspond (with
multiplicities) to the zeros of $\cR_{\Delta,\cB\cap V'}(P_n)$ in $D_z$. We want to
count those zeros contained in $D_z^2$. Recall the following
consequence of Jensen's formula \cite{iy:jensen}.

\begin{Prop}\label{prop:jensen}
  Let $f:\bar D\to\C$ be holomorphic. Denote by $M$ (resp. $m$) the maximum
  of $|f(z)|$ on $\bar D$ (resp. $\bar D^2$). Then there exists a
  constant $C$ such that
  \begin{equation}
    \#\{z\in D^2 : f(z)=0 \} \le C\cdot\log\frac M m.
  \end{equation} 
\end{Prop}

We apply this proposition to $P_n$ in $D_z$.  We first note that $M$
is a product of $|P_n|$ evaluated at $\mu$ points
$p_1,\ldots,p_\mu\in\B$. It is clear that
$\log |P_n(z_j)|\le\poly(\delta(P_n))$, so $\log M$ is appropriately
bounded.

It remain to show that $\log(1/m)$ is appropriately bounded. Let
$p_1,\ldots,p_\mu$ denote the points of $V'$ lying over the origin in
$\Delta'$. Consider the multiplicity operators
$\mo(P_1,\ldots,P_{n-1})$ with respect to the direction of the
$\vw$-coordinates (which we think of as a leaf of the foliated space
$L_{n-1}\M$). By Corollary~\ref{cor:mo-v-dist}, at every point $p_j$
there is such a multiplicity operator with
$\log|1/\mo_{p_j}(P_1,\ldots,P_{n-1})|$ appropriately bounded in
absolute value (here we use the fact that we perturbed to general
position). According to Lemma~\ref{lem:mo-v-wp} each point $p_j$ is
the center of a Weierstrass polydisc $\Delta_j$ \emph{in the same
  coordinate system}, and with the logarithms of all radii
appropriately bounded in absolute value.

Denote by $\cR_j:=\cR_{\Delta_j,\cB\cap V'}(P_n)$. The domains of all
these functions (and of $\cR$ itself) contain a disc $D$ of radius
$r$, with $\log (1/r)$ appropriately bounded. Note that
\begin{equation}
  |\cR(z)| \ge \frac{\prod_{j=1,\ldots,\mu} |\cR_j(z)|}{e^{\poly(\delta_\vxi,\delta_V)}}
\end{equation}
since the numerator contains the value of $P_n$ evaluated at every
point of $\cB\cap V'$ over $z$ (possibly more than once), and these
evaluations are always bounded from above by $e^{\poly(\delta(P_n))}$
as we have seen above. It will therefore suffice to find a point in
$D$ where $\log (1/|\cR_j|)$ is appropriately bounded for every
$j$. For this we use Lemma~\ref{lem:mo-v-resultant}. Namely, the lemma
shows that $\log(1/|\cR_j|)$ is appropriately bounded outside a union
of balls of total radius smaller than $r/\mu$, and taking union over
$j=1,\ldots,\mu$ one can find a point where this happens
simultaneously for every $j$. This shows that $\log(1/m)$ is
appropriately bounded and concludes the proof of
Theorem~\ref{thm:zero-count}.

\subsection{Proof of Corollary~\ref{cor:zero-count-any-codim}}

This follows immediately by applying Proposition~\ref{prop:CI-choose}
with $m=n$ and applying Theorem~\ref{thm:zero-count} to the $W$ that
one obtains.

\subsection{Proof of Theorem~\ref{thm:wp-cover}}

We will prove Theorem~\ref{thm:wp-cover} in dimension $n$ assuming
that Theorem~\ref{thm:zero-count} holds in smaller dimensions. It will
be enough to find a Weierstrass polydisc $\Delta\subset\cB$ around the
origin containing a ball of radius $r$ such that $1/r$ and
$e(\V\cap\cB,\Delta)$ are appropriately bounded. Indeed, if we can do this
then by a simple rescaling and covering argument we can find a
collection of polydiscs covering $\cB^2$.

According to Corollary~\ref{cor:wp-v-vol} it will be enough to show
that $\vol(\cB\cap V)$ is appropriately bounded. This volume can be
estimated using complex integral geometry in the spirit of Crofton's
formula. Namely, according to \cite[Proposition~14.6.3]{chirka} we
have
\begin{equation}\label{eq:complex-crofton}
  \vol(V\cap\cB) = \const(n) \int_{G(n,\codim V)} \#(V\cap\cB\cap L) \d L
\end{equation}
where $G(n,k)$ denotes the space of all $k$-dimensional linear
subspaces of $\C^n$ with the standard measure.

We now pass to general position with respect to $V$ as
in~\secref{sec:perturb-main}. Since our reparametrizing map can be
assumed to be close to the identity, this does not change the volume
by a factor of more than (say) two. Hence it is enough to estimate the
volume in the new coordinates, and by~\eqref{eq:complex-crofton} it
will suffice to show that $\#(V\cap\cB\cap L)$ is appropriately
bounded for every $\vxi$-linear subspace of dimension
$k=\codim V$. Since the $\cB\cap L$ are all unit balls in leafs of
$L_k\M$, the result now follows by the inductive application of
Theorem~\ref{thm:zero-count} (using the fact that $L_k\M$ has no
new unlikely intersections with $V$).

\section{Proof of Theorem~\ref{thm:alg-count}}
\label{sec:alg-count-proof}

We start by developing some general material on interpolation of
algebraic points in Weierstrass polydiscs. It is convenient to state
these results in the general analytic context without reference to
foliated spaces, and we take this viewpoint
in~\secref{sec:interpolation}. In~\secref{sec:alg-count-proof-final}
we finish the proof of Theorem~\ref{thm:alg-count}.

\subsection{Interpolating algebraic points}
\label{sec:interpolation}

Let $n\in\N$. The asymptotic constants in this section will depend
only on $n$. Let $\Delta=\Delta_\vx\times\Delta_\vw\subset\C^n$ be a
Weierstrass polydisc for an analytic set $X\subset\C^n$ of pure
dimension $m$. Let $F\in\cO(\bar\Delta)$. Let $\cM\subset\N^n$ be the
set
\begin{equation}
  \cM := \N^m\times\{0,\ldots,e(X,\Delta)-1\}^{n-m}.
\end{equation}
We also set $E:=e(X,\Delta)^{n-m}$. Recall the following result
combining \cite[Theorem~3]{me:rest-wilkie} and
\cite[Proposition~8]{me:analytic-interpolation}.

\begin{Prop} \label{prop:wp-expansion}
  On $\Delta^2$ there is a decomposition
  \begin{equation}
    F = \sum_{\alpha\in\cM} c_{\alpha} \vx^\alpha + Q, \qquad Q\in\cO(\Delta^2),
  \end{equation}
  where $Q$ vanishes on $\Delta^2\cap X$ and
  \begin{equation}
    \norm{c_\alpha \vx^\alpha}_{\Delta^2} = O(2^{-|\alpha|} \cdot \norm{F}_\Delta).
  \end{equation}
\end{Prop}

We now fix $\Phi\in\cO(\Delta)^{m+1}$.  In \cite{me:rest-wilkie}
Proposition~\ref{prop:wp-expansion} was used in combination with the
interpolation determinant method of Bombieri and Pila
\cite{bombieri-pila} to produce an algebraic hypersurface
interpolating the points of $X\cap\Delta^2$ where $\Phi$ takes
algebraic values of a given height in a fixed number field. However,
this method does not produce good bounds when one considers the more
general $[X\cap\Delta^2](g,h;\Phi)$ where the number field may
vary. Instead we will use an alternative approach proposed by Wilkie
\cite{wilkie:pw-notes}, which is based on the following variant of the
Thue-Siegel lemma. This idea was used in a slightly different context
by Habegger in \cite{habegger:approx}.

\begin{Lem}[\protect{\cite[Lemma 4.11]{waldschmidt:book}}]\label{lem:thue-siegel}
  Let $A\in\Mat_{\mu\times\nu}(\R)$. For any $N\in\N$ there exists a
  vector $\vv\in\Z^\nu\setminus\{0\}$ satisfying
  \begin{equation}
    \begin{aligned}
      \norm{\vv}_\infty&\le N+1, & \norm{A\vv}_\infty &\le N^{\frac{\mu-\nu}\mu} \norm{A}_\infty.
    \end{aligned}
  \end{equation}
\end{Lem}

By combining Proposition~\ref{prop:wp-expansion} and
Lemma~\ref{lem:thue-siegel} we obtain the following.

\begin{Lem}\label{lem:interpolation-upper-bd}
  Let $d,N\in\N$. There exists a polynomial
  $P\in\Z[y_1,\ldots,y_{m+1}]\setminus\{0\}$ with $\deg P\le d$ and
  all coefficients bounded in absolute value by $N$, such that
  \begin{equation}
    \norm{(P\circ\Phi)\rest{X\cap\Delta^2}} \le \poly(d) E \norm{\Phi}_\Delta^d (N \log N) 2^{-d(E^{-1}\log N)^{\frac1{m+1}}}.
  \end{equation}
\end{Lem}
\begin{proof}
  Let $\Phi^\alpha$ for $\alpha\in\N^{m+1}$ denote the monomial in
  the $\Phi$ variables with the usual multiindex notation. Note that
  $\norm{\Phi^\alpha}\le\norm{\Phi}^{|\alpha|}$. For each $|\alpha|\le d$ apply
  Proposition~\ref{prop:wp-expansion} to $\Phi^\alpha$ to get
  \begin{equation}
    \Phi_\alpha = \sum_{\beta\in\cM} c_{\alpha,\beta} \vx^\beta + Q, \qquad Q\in\cO(\Delta^2),
  \end{equation}
  where $Q$ vanishes on $\Delta^2\cap X$ and
  \begin{equation}\label{eq:Phi-expansion-bd}
    \norm{c_{\alpha,\beta} \vx^\beta}_{\Delta^2} = O(2^{-|\beta|} \cdot \norm{\Phi}_\Delta^d).
  \end{equation}

  Fix $k\in\N$ to be chosen later. Using Lemma~\ref{lem:thue-siegel}
  we find a linear combination
  $\sum_{|\alpha|\le d}v_\alpha \Phi^\alpha$ with $v_\alpha$ integers
  and $|v_\alpha|<N$, not all zero, such that for every $|\beta|\le k$
  we have
  \begin{equation}\label{eq:thue-estimate}
    \norm{\textstyle{\sum}_{|\alpha|\le d} v_\alpha c_{\alpha,\beta}x^\beta}_{\Delta^2} = O(\norm{\Phi}_\Delta^d)\cdot N^{\frac{\mu-\nu}\mu},
    \qquad \text{where}\quad
    \begin{aligned}
      \mu &\sim E k^m \\
      \nu &\sim d^{m+1}.
    \end{aligned}
  \end{equation}
  We now write
  \begin{equation}
    \sum_{|\alpha|\le d} v_\alpha \Phi^\alpha =
    \sum_{|\beta|\le k}\sum_{|\alpha|\le d} v_\alpha c_{\alpha,\beta}x^\beta
    +\sum_{|\beta|>k} \sum_{|\alpha|\le d} v_\alpha c_{\alpha,\beta}x^\beta  =A+B.
  \end{equation}
  For $A$ we have by~\eqref{eq:thue-estimate} the estimate
  \begin{equation}
    A \le O(E k^m \norm{\Phi}_\Delta^d)\cdot N^{\frac{\mu-\nu}\mu},
  \end{equation}
  and for $B$ we have by~\eqref{eq:Phi-expansion-bd} the estimate
  \begin{equation}
    B \le O\big(N d^{m+1} \norm{\Phi}_\Delta^d \sum_{|\beta|>k} 2^{-\beta}\big) = O(N d^{m+1}\norm{\Phi}_\Delta^d2^{-k}).
  \end{equation}
  Choosing $k=d(E^{-1}\log N)^{1/(m+1)}$ proves the lemma.
\end{proof}

We will compare the upper bound of
Lemma~\ref{lem:interpolation-upper-bd} with the following elementary
lower bound at points where $\Phi$ takes algebraic values of bounded
height and degree.

\begin{Lem}[\protect{\cite[Lemma~14]{habegger:approx}}]\label{lem:algebraic-lower-bd}
  Let $P\in\Z[y_1,\ldots,y_{m+1}]$ be a polynomial of degree $d$ and
  all coefficients bounded in absolute value by $N$. Suppose that
  $\vy\in(\Qa)^{m+1}$ and $P(\vy)\neq0$. Then
  \begin{equation}
    |P(\vy)| \ge \big(d^{m+1} N H(\vy)^{d(m+1)}\big)^{-[\Q(\vy):\Q]}.
  \end{equation}
\end{Lem}

For a subset $A\subset\C^n$ we denote
\begin{equation}
  A(g,h;\Phi) := \{ p\in A : [\Q(\Phi(p)):\Q] \le g \text{ and } h(\Phi(p))\le h\}.
\end{equation}
We now come to our interpolation result.

\begin{Prop}\label{prop:interpolation}
  The set $[\Delta^2\cap X](g,h;\Phi)$ is contained in the zero locus
  of $P\circ\Phi$, where $P\in\Z[y_1,\ldots,y_{m+1}]\setminus\{0\}$
  and
  \begin{align}
    \deg P&\sim g\cdot E\cdot (gh+\log\norm{\Phi}_\Delta)^m & h(P) &\sim E\cdot(gh+\log\norm{\Phi}_\Delta)^{m+1}.
  \end{align}
\end{Prop}
\begin{proof}
  Let $d,N\in\N$ and construct the polynomial $P$ as in
  Lemma~\ref{lem:interpolation-upper-bd}. At any point
  $\vx\in[\Delta^2\cap X](g,h;\Phi)$, if $P\circ\Phi(x)\neq0$ then
  \begin{multline}
    [\poly(d)\cdot N \cdot 2^{(m+1)hd}]^{-g} \le 
    |P\circ\Phi(\vx)| \le \\
    \poly(d) E \norm{\Phi}_\Delta^d (N \log N) 2^{-d(E^{-1}\log N)^{\frac1{m+1}}}
  \end{multline}
  and hence
  \begin{equation}\label{eq:interpolation-main-ineq}
    2^{-O(\frac{g \log N}d + gh +\log\norm{\Phi}_\Delta+\frac{\log E}d)} \le 2^{-(E^{-1} \log N)^{\frac1{m+1}}}.
  \end{equation}
  Now choose
  \begin{align}
    d&=C^{m+1}E (gh+\log\norm{\Phi}_\Delta)^m\cdot g \\
    \log N&=C^{m+1}E(gh+\log\norm{\Phi}_\Delta)^{m+1}.
  \end{align}
  Then~\eqref{eq:interpolation-main-ineq} becomes
  \begin{equation}
    2^{-O(gh +\log\norm{\Phi}_\Delta)} \le 2^{-C(gh +\log\norm{\Phi}_\Delta)}
  \end{equation}
  which is impossible for a sufficiently large constant $C=C(m)$, and
  we deduce that $P\circ\Phi$ vanishes on $[\Delta^2\cap X](g,h;\Phi)$
  as claimed.
\end{proof}

\subsection{Finishing the proof of Theorem~\ref{thm:alg-count}}
\label{sec:alg-count-proof-final}

Theorem~\ref{thm:alg-count} follows immediately from the following
inductive step, where we start with $\cW=\C^\ell$ and proceed until
$\dist(\cB,\Sigma(V,\cW))<\e$.

\begin{Prop}\label{prop:alg-count-step}
  Let $\cW\subset\C^\ell$ be an irreducible $\Q$-variety of positive
  dimension. Suppose that $\e:=\dist(\cB,\Sigma(V,W))$ is
  positive. Then there exists a collection of irreducible
  $\Q$-subvarieties $\{\cW_\alpha\subset\cW\}$ of codimension one such
  that
  \begin{equation}
    \cW\cap A(g,h) \subset \cup_\alpha \cW_\alpha
  \end{equation}
  and
  \begin{equation}\label{eq:alg-count-step-deg}
    \#\{\cW_\alpha\}, \max_\alpha \delta_{\cW_\alpha} =
    \poly(\delta_\vxi,\delta_V,\delta_\cW,\delta_\Phi,g,h,\log \e^{-1}).
  \end{equation}
\end{Prop}

\begin{proof}[Proof of Proposition~\ref{prop:alg-count-step}]
  Let $m:=\dim\cW$. Set $V'=V\cap\Phi^{-1}(\cW)$. We claim that
  \begin{equation}
    \{ p: \dim(V'\cap\cL_p)\ge m \} \subset \Sigma(V,\cW;\Phi).
  \end{equation}
  Indeed, if $p\not\in\Sigma(V,\cW;\Phi)$ then $\Phi_{V\cap\cL_p}$ is
  finite. If $V'\cap\cL_p$ has a component $C$ of dimension at least
  $m$ then $\dim\Phi(C)\ge m$ and $C\subset\cW$, so $\Phi(C)$ is a
  component of $\cW$ contradicting the definition of
  $\Sigma(V,\cW;\Phi)$.
  
  Using Proposition~\ref{prop:CI-choose} we find a
  complete-intersection $W$ of codimension $n-m+1$ containing $V'$ and
  satisfying $\Sigma_W\subset\Sigma(V,\cW;\Phi)$ with appropriate
  control over $\delta_W$. Using Theorem~\ref{thm:wp-cover} we cover
  $\cB^2$ by sets $\Delta_\beta^2$ where $\Delta_\beta$ is a
  Weierstrass polydisc for $\cB\cap W$ and $\#\{\Delta_\beta\}$ and
  $e(\cB\cap W,\Delta_\beta)$ are bounded as
  in~\eqref{eq:alg-count-step-deg}.
  
  Choose a set of $m$ coordinates $S\subset\{1,\ldots,\ell\}$ such
  that the projection of $\cW$ to these coordinates is dominant. Using
  Proposition~\ref{prop:interpolation} we construct a polynomial
  $P_\beta\in\Z[y_1,\ldots,y_\ell]\setminus\{0\}$ depending only on
  the variables $\{y_s\}_{s\in S}$ such that
  $\delta(P_S)=\poly(g,h,\delta_\Phi)$ and $P_\beta\circ\Phi$ vanishes
  identically on $[\Delta_\beta^2\cap W](g,h;\Phi)$. Finally taking
  $\{\cW_\alpha\}$ to be the union of the collection of irreducible
  components of $\cW\cap\{P_\beta=0\}$ for every $\beta$ proves the
  claim.
\end{proof}

\section{Diophantine applications}
\label{sec:applications}

Theorem~\ref{thm:alg-count} gives, under suitable conditions, an
effective polylogarithmic version of the counting theorem of Pila and
Wilkie \cite{pila-wilkie}. The counting theorem has found numerous
applications in various problems of Diophantine geometry, and our
principal motivation in pursuing Theorem~\ref{thm:alg-count} is the
potential for effectivizing these applications. In this section we
illustrate how this can be achieved for two of the influential
applications of the counting theorem: Masser-Zannier's finiteness
result for simultaneous torsion points on elliptic squares
\cite{mz:torsion-annalen} and Pila's proof of the Andr\'e-Oort
conjecture for modular curves \cite{pila:ao}. Each of these directions
have led to significant progress and numerous additional results, many
of which seem to be amenable to the same ideas. We also prove a Galois
orbit lower-bound for torsion points on elliptic curves following an
idea of Schmidt. We focus on the most basic examples in each of these
directions to present the method in the simplest context, and will
address some of the more involved applications separately in the
future.

\subsection{Simultaneous torsion points}

Let $T\subset\C^4\times(\C\setminus\{0,1\})$ denote the fibered
product of two copies of the Legendre family,
\begin{equation}
  T = \{(x_1,y_1,x_2,y_2,\lambda) : y_j^2=x_j(x_j-1)(x_j-\lambda), j=1,2\}.
\end{equation}
The fiber of $T$ over $\lambda$ is an elliptic square
$E_\lambda\times E_\lambda$, and we use the additive notation for the
group law on this scheme. We will also write $P=(x_1,y_1)$ and
$Q=(x_2,y_2)$.

\begin{Thm}\label{thm:effective-mz}
  Let $C\subset T$ be an irreducible curve over a number field $\K$
  with non-constant $\lambda$. Suppose that no relation $nP=mQ$ holds
  identically on $C$, for any $(n,m)\in\N^2\setminus\{0\}$. Then at
  any point $c\in C$ where $P(c),Q(c)$ are both torsion, their
  corresponding orders of torsion are effectively bounded by
  $\poly(\delta_C,[\K:\Q])$.
\end{Thm}

The proof is given in~\secref{sec:effective-mz-proof}.
Theorem~\ref{thm:effective-mz} implies the finiteness of the set of
simultaneous torsion points, which is the main statement of
\cite{mz:torsion-annalen}. It also implies that the set of
simultaneous torsion points is effectively computable in polynomial
time: for each possible torsion order $k$ up to bound provided in the
theorem, one can compute the algebraic equations
$(P^k,Q^k,c)=(\infty,\infty,c)$ using the group law on $T$, intersect
with the equation defining $C\subset T$, and use elimination theory or
Grobner base algorithms to compute the sets of solutions $c$.

We remark that numerous variations on the theme of
Theorem~\ref{thm:effective-mz} have been studied by Masser-Zannier
\cite{mz:torsion-ajm,mz:torsion-adv,mz:torsion-pell-jems,mz:pell-integration}
and by Barroero-Capuano
\cite{bc:relations-in-powers,bc:unlike-elliptic-multiplicative,bc:pell-plms}
and Schmidt \cite{schmidt:relative-mm}. These include very interesting
applications to the solvability of Pell's equation in polynomials and
to integrability in elementary terms. Effective bounds for these
contexts, analogous to Theorem~\ref{thm:effective-mz}, should in
principle provide the last step toward effective solvability of these
classical problems. While we do not address these generalizations
directly in this paper, they do appear to be similarly amenable to our
methods. We have developed some of the material (most specifically the
growth estimates in Appendix~\ref{appendix:fuchs-growth}) with an eye
to treating the more general types of period maps arising in these
applications.

\subsection{Andr\'e-Oort for modular curves}

We refer the reader to \cite{pila:ao} for the general terminology
related to the Andr\'e-Oort conjecture in the context of $\C^n$. We
will prove the following.

\begin{Thm}\label{thm:ao-decidable}
  Let $V\subset\C^n$ be an algebraic variety over a number field $\K$.
  Then the degrees of all maximal special subvarieties, as well as the
  discriminants of all their special coordinates, are bounded by
  $\poly_n(\delta_V,[\K:\Q])$. Here \emph{the implied constant is not
    effective}. Moreover there exists an algorithm that computes the
  collection of all maximal special subvarieties of $V$ in
  $\poly_n(\delta_V,[\K:\Q])$ steps.
\end{Thm}

The proof is given in~\secref{sec:effective-ao-proof}. Note that this
is the only point in the present paper where the implied asymptotic
constant is not effectively computable in principle. The constants
depend on Siegel's asymptotic lower bound for class numbers, and
obtaining an effective form of this bound is a well-known and deep
problem. Effectivity of this universal constant notwithstanding,
Theorem~\ref{thm:ao-decidable} still establishes the polynomial-time
decidability of the Andr\'e-Oort conjecture in $\C^n$ for fixed
$n$. We also note that the constants do depend effectively on $n$, so
the result also establishes the decidability of Andr\'e-Oort for
$\C^n$ with $n$ considered as a variable. We remark that the
Andr\'e-Oort conjecture for more general products of modular curves
can be proved by reduction to the $\C^n$ case, and this certainly
preserves effectivity, but we do not pursue the details of this here.

\subsection{A Galois-orbit lower bound for torsion points}

We will prove the following.

\begin{Thm}\label{thm:torsion-deg}
  Let $E$ be an elliptic curve defined over a number field $\K$, and
  $p\in A$ a torsion point of order $n$. Then
  \begin{equation}
     n=\poly_g([\K:\Q],\hFal(E),[\K(p):\K]).
  \end{equation}
\end{Thm}

The proof is given
in~\secref{sec:degree-bounds-proof}. Theorem~\ref{thm:torsion-deg} is
not new: it follows (with more precise dependence on the parameters)
from the work of David \cite{david:torsion-deg-elliptic}. It has also
been generalized to abelian varieties of arbitrary genus under some
mild conditions \cite{david:torsion-deg-abelian}, see also
\cite{mz:pell-integration} for the general case. The proof presented
here is different, replacing the use of transcendence methods by point
counting using an idea of Schmidt.

We restrict our formal presentation to the elliptic case as the
general case requires some additional technical tools that we do not
treat in this paper. However we sketch in~\secref{sec:bound-any-g} how
the proof extends to arbitrary genus (we restrict to principally
polarized abelian varieties and have not considered the general case).
We also mention further implications for Galois orbit lower bounds in
Shimura varieties in~\secref{sec:deg-bounds-future}.

\section{Proof of Theorem~\ref{thm:effective-mz}}
\label{sec:effective-mz-proof}

To simplify our presentation we will assume everywhere that $\K=\Q$,
but the proof is essentially the same in the general case.

\subsection{The foliation}
\label{sec:simul-tor-foliation}

We will construct a one-dimensional foliation encoding for each
$c\in C$ a pair of lattice generators $(f,g)$ for the curve
$E_{\lambda(c)}$ and a pair of elliptic logarithms $z,w$ for the
points $P(c),Q(c)\in E_{\lambda(c)}$. This can be done with the help
of the classical Picard-Fuchs differential operator as follows.

We will work in the space over $\C$ given by
\begin{equation}
  \M := C\times G, \qquad G:=(\Mat_{2\times2},+)\rtimes(\GL_2,\cdot)
\end{equation}
where we will use the matrix $M_L$ (resp. $M_P$) to denote the
coordinate on the second (resp. third) factor, and more specifically
write
\begin{align}
  M_L =& \begin{pmatrix}
    z & w \\ \dot z & \dot w    
  \end{pmatrix},
  &
  M_P =& \begin{pmatrix}
    f & g \\ \dot f & \dot g
  \end{pmatrix}.
\end{align}
We consider $G$ as a semidirect product with respect to the left
action of $\GL_2$ on $\Mat_{2\times 2}$ given by
$M_P\cdot M_L=M_P M_L$, i.e. with the product rule
\begin{equation}
  (M_L,M_P)(M_L',M_P') = (M_L+M_P M_L',M_P M_P').
\end{equation}

Let $\Sigma\subset\C_\lambda$ denote the set consisting of $0,1$, the
critical values of $\lambda\rest C$, and the points where
$\lambda=x_1(\lambda)$ or $\lambda=x_2(\lambda)$ (cf.
\cite[p.459]{mz:torsion-annalen} where a similar choice is made). We
set $A_\lambda=\C\setminus\Sigma$ and replace $C$ by the part of $C$
that lives over $A_\lambda$.

We define will take our foliation $\cF$ to be generated by a vector
field
\begin{equation}
  \xi := \pd{}\lambda + \pd{x_1}{\lambda}\pd{}{x_1}+\cdots+\pd{\dot w}{\lambda}\pd{}{\dot w}
\end{equation}
where we will show below how to express each of the
$\pd{}\lambda$-derivatives of the coordinates as regular functions on
$\M$.

We start with the coordinates of $C$. Since we assume $\lambda\rest C$
is submersive there is a unique lift of $\pd{}\lambda$, thought of as
a section of $T(A_\lambda)$, to a section $\xi_C$ of $T(C)$. The
coordinates of this section are regular functions, and their height
and degree can be readily estimated e.g. by writing out $T(C)$
explicitly as a Zariski tangent bundle. The $\pd{}{x_j}$ and
$\pd{}{y_j}$ coordinates of $\xi_C$ give our $\pd{x_j}{\lambda}$ and
$\pd{y_j}{\lambda}$

We now turn to the equations for $(f,g)$. Recall that each elliptic
period
\begin{equation}\label{eq:elliptic-period}
  I(\lambda) := \oint_{\delta(\lambda)}\omega, \qquad \omega=\frac{\d x}y
\end{equation}
where $\delta(\lambda)\in H_1(E_\lambda)$ is a continuous family
satisfies the Picard-Fuchs equation
\begin{equation}\label{eq:period-PF}
  L I(\lambda)=0, \qquad L=\lambda(1-\lambda)\pd{^2}{\lambda^2}+(1-2\lambda)\pd{}\lambda-\frac14.
\end{equation}
We encode the fact that $f$ satisfies this second order equation by
requiring
\begin{align}
  \pd{}\lambda f &= \dot f & \pd{}\lambda \dot f = \frac{(1/4)f - (1-2\lambda)\dot f}{\lambda(1-\lambda)}.
\end{align}
Note that $\lambda(1-\lambda)$ is invertible on $A_\lambda$. We impose
the same equations on $(g,\dot g)$.

Finally, to handle $z,w$, recall that each elliptic logarithm
\begin{equation}\label{eq:elliptic-logarithm}
  \hat I(\lambda) := \int_\infty^{P(\lambda)} \omega
\end{equation}
satisfies an inhomogeneous Picard-Fuchs equation. More explicitly,
applying the operator $L$ to $\hat I(\lambda)$ we obtain by direct
computation
\begin{equation}\label{eq:logarithm-PF}
  L\hat I(\lambda) = B+\int_\infty^{P(\lambda)} L \omega =B+\int_\infty^{P(\lambda)} \frac12 \d\left(\frac{y}{(x-\lambda)^2}\right)\\
  = B+\frac12 \frac{y_1(\lambda)}{(x_1(\lambda)-\lambda)^2}
\end{equation}
where $B$ denotes the terms coming for the derivation of the boundary
points, e.g. $\omega(P(\lambda)')$ for the first derivative. To make
this computation explicitly write $y:=\sqrt{x(x-1)(x-\lambda)}$ as a
function as a function of $x,\lambda$, express the integral as a path
integral in the $x$-plane, and use the usual derivation rules.

Denote the right hand side of~\eqref{eq:logarithm-PF} by $R_z$. Then $R_z$
is a regular function on $\M$ by our definition of $A_\lambda$, and
the explicit derivation readily shows that
$\delta_R=\poly(\delta_C)$. We may thus write the equations for $z$ as
\begin{align}
  \pd{}\lambda z &= \dot z &
  \pd{}\lambda \dot z = \frac{(1/4)z - (1-2\lambda)\dot z+R_z}{\lambda(1-\lambda)}.
\end{align}
We impose the a similar set of equations on $(w,\dot w)$, with the
right hand side $R_w$.

As a consequence of this construction, one leaf $\cL_0$ of our
foliation is given (locally) by the graph over $C$ of $(f,g,z,w)$
where $f,g$ are taken to be the two generators of the lattice
$E_\lambda$, and $z,w$ are taken to be elliptic logarithms of
$P(\lambda),Q(\lambda)$. As one analytically continues this leaf
$\cL_0$ obtains other choices for the generators $f,g$ and the
logarithms $z,w$.

We will also require a description of the remaining leafs. This is
fairly simple to obtain: our equations for $f,g$ are equivalent to the
Gauss-Manin linear equations $Lf=Lg=0$. For the standard leaf $\cL_0$
these are taken to be two linearly independent solutions, and any
other solution is obtained by replacing $M_P$ by $M_PG_P$ for some
$G_P\in\GL_2(\C)$. Similarly the equations for $z,w$ are equivalent to
$Lz=R_z,Lw=R_w$, and since $f,g$ form a basis of solutions of the
homogeneous equations on any leaf, any other leaf with the same $f,g$
is obtained by replacing $M_L$ by $M_L+M_PG_L$ for some
$G_L\in\Mat_{2\times2}(\C)$. In other words $\cF$ is a flat
structure of the principal $G$-bundle $\M$, where $G$ acts on itself
by multiplication on the right.

\subsection{Degree and height bounds}

We need two lemmas from \cite{mz:torsion-annalen} on the degree and
height of points $c\in C$ where either $P$ or $Q$ is torsion.

\begin{Lem}[\protect{\cite[Lemma~7.1]{mz:torsion-annalen}}]\label{lem:torsion-degree}
  Let $c\in C$ be such that $P(c)$ or $Q(c)$ is torsion of order
  $n$. Then
  \begin{equation}
    n \le \poly(\delta_C,[\Q(\lambda(c)):\Q],h(\lambda(c))).
  \end{equation}
\end{Lem}
\begin{proof}
  This follows at once from the proof of
  \cite[Lemma~7.1]{mz:torsion-annalen}, where one just needs to track
  down the constant $c$ to find that it is $c=\delta_C$. We also give
  an independent proof in~\secref{sec:degree-bounds-proof}.
\end{proof}

\begin{Lem}[\protect{\cite[Lemma~8.1]{mz:torsion-annalen}}]\label{lem:torsion-height}
  Let $c\in C$ be such that $P(c)$ or $Q(c)$ is torsion. Then
  \begin{equation}
    h(\lambda(c)) \le \poly(\delta_C).
  \end{equation}
\end{Lem}
\begin{proof}
  Without the explicit dependence on $\delta_C$ this is
  \cite[Lemma~8.1]{mz:torsion-annalen}. The dependence on $\delta_C$
  can be seen from the proof of
  \cite[Proposition~3.1]{zannier:book}. Specifically it comes down to
  Zimmer's estimate for the difference between the Neron-Tate height
  $\hat h(P)$ and Weil height $h(P)$ in the function field case, where
  the explicit form given in \cite[p.40,~Theorem]{zimmer} shows that
  the asymptotic constants are $\poly(\delta_C)$.
\end{proof}

Recall that we defined $A_\lambda:=\C\setminus\Sigma$ for some finite
set $\Sigma$. For $\delta>0$ we define
$\Lambda_\delta\subset A_\lambda$ as
\begin{equation}
  \Lambda_\delta := \{ \lambda : |\lambda|<\delta^{-1},\ \forall\sigma\in\Sigma: |\lambda-\sigma|>\delta \}.
\end{equation}
We record a consequence of Lemma~\ref{lem:torsion-height}.

\begin{Lem}{\protect{\cite[Lemma~8.2]{mz:torsion-annalen}}}\label{lem:orbit-distribution}
  Let $\lambda\in A_\lambda$. Then for
  $\delta=2^{-\poly(\delta_C,h(\lambda))}$ at least half of the Galois
  conjugates of $\lambda$ are in $\Lambda_\delta$.
\end{Lem}
\begin{proof}
  The proof is the same as
  \cite[Lemma~8.2]{mz:torsion-annalen}. Briefly, we have an upper
  bound on the heights of $\lambda$ and $\lambda-\sigma$ for
  $\sigma\in\Sigma$, and this means that averaging over the Galois
  orbit none of these can be too small (or too big) in absolute value.
\end{proof}

\subsection{Setting up the domain for counting}
\label{sec:simul-torsion-domain}

Let $c\in C$ be such that $P(c),Q(c)$ are both torsion, and let $n$
denote the maximum among their orders of torsion and
$N(c):=[\Q(c):\Q]$. According to Lemma~\ref{lem:torsion-height} we
have $h(\lambda(c))=\poly(\delta_C)$. Then by
Lemma~\ref{lem:orbit-distribution} at least half of the Galois orbit
of $\lambda(c)$ lies in a set $\Lambda_\delta$ with some
$\delta=2^{-\poly(\delta_C)}$. Moreover
\begin{equation}\label{eq:n-vs-Nc}
  n=\poly(\delta_C,N(c))
\end{equation}
by Lemma~\ref{lem:torsion-degree}.

We choose a collection of $\poly(\delta_C)$ discs
$D_i\subset A_\lambda$ such that
\begin{align}
  D_i^{1/4}&\subset\Lambda_{\delta/2}, & \Lambda_\delta&\subset\cup_i D_i.
\end{align}
This is possible by elementary plane geometry using a logarithmic
subdivision process. For example, it is enough to show that for each
$r>0$, one can make such a choice of discs $D_i$ with
$D_i^{1/4}\subset\Lambda_{r/2}$ to cover
$\Lambda_r\setminus\Lambda_{2r}$. This is equivalent, after rescaling
by $r$, to proving the same fact for $r=1$, and here the number of
discs $D_i$ is easily seen to depend polynomially on the number of
points in $\Sigma$.

In conclusion, we proved the following.
\begin{Lem}\label{lem:branch-choice}
  There exists one disc $D=D_i$, and one branch of the curve $C$ over
  $D_i$, such that the number of Galois conjugates $c_\sigma$ with
  $\lambda(c_\sigma)\in D_i$ and $(P(c_\sigma),Q(c_\sigma))$ in the
  chosen branch of $C$ is at least $N(c)/\poly(\delta_C)$.
\end{Lem}

\subsection{Growth estimates for the leaf}
\label{sec:leaf-growth}

We will consider the ball $\cB$ in $\cL_0$ corresponding to $D^{1/2}$
in the $\lambda$-coordinate, with the $P,Q$ coordinates corresponding
to the branch of $C$ chosen in Lemma~\ref{lem:branch-choice}. To apply
Theorem~\ref{thm:alg-count} we must estimate the radius of the ball
$\B_R$ containing this leaf. This can possibly be done by hand for the
elliptic case treated in this paper, but we give a more general
approach using growth estimates for differential equations which seems
easier to carry out in more general settings.

\begin{Rem}
  The main difficulty is to obtain appropriate estimates for the
  elliptic logarithms $z,w$. These are given by incomplete elliptic
  integrals. In the early examples considered by Masser-Zannier, these
  endpoints were taken to have a constant $x$-coordinates, say
  $x=2,3$. In such cases the incomplete integrals can be estimated in
  a straightforward manner.

  When one considers an arbitrary curve $C$, the integration endpoints
  vary with $c\in C$. It is then necessary to carefully choose the
  integration path to avoid passing near singularities, and to track
  how the integration path is deformed as one analytically continues
  over a domain in $C$. In general, throughout such a deformation the
  length of the integration path may unavoidably grow as it picks up
  copies of vanishing cycles by the Picard-Lefschetz
  formula. Effectively controlling this phenomenon in terms of the
  degree and height of $C$ already appears fairly difficult to do by
  hand.
\end{Rem}

We start with the coordinates $P,Q$. Since
\begin{equation}
  \log \dist^{-1}(D^{1/2},\Sigma) = \poly(\delta_C)
\end{equation}
one can check that the coordinates $P,Q$ are bounded by
$e^{\poly(\delta_C)}$. For instance one may use the general effective
bounds for semialgebraic sets proved in \cite{basu:bounds}, though for
this special case much more elementary arguments would suffice. We
proceed to consider the remaining coordinates, which are given by
(transcendental) elliptic integrals and require a more delicate
approach.

Consider first the elliptic periods $f,g$. Fix some
$\lambda_0\in\C\setminus\{0,1\}$, say $\lambda_0=1/2$. For some fixed
choice of the integration paths staying away from
$0,1,\lambda_0,\infty$, we can directly estimate
\begin{equation}
  |f|,|\dot f|,|g|,|\dot g|,\frac1{\Im(f/g)} < M_0
\end{equation}
at $\lambda=\lambda_0$ with $M_0$ an effective constant. Indeed for
such a path the integrals are nicely convergent and one can
approximate them up to any given precision effectively and find such a
constant. Our goal is to deduce an effective estimate for these
quantities after analytic continuation from $\lambda_0$ to $D^{1/2}$.

Recall that $f,g$ satisfy the Picard-Fuchs differential
equation~\eqref{eq:period-PF}. Since this is a Fuchsian equation, the
theorem of Fuchs \cite[Theorem~19.20]{sergei:book} implies that $f,g$
(and their derivatives) grow polynomially as one approaches the
singular locus of the operator (here $\lambda=0,1,\infty$) along
geodesic lines on $\P^1$. In Appendix~\ref{appendix:fuchs-growth} we
prove an effective version of this theorem. Specifically, using
Theorem~\ref{thm:fuchs-growth} we get for any $\lambda\in D^{1/2}$ the
estimate
\begin{equation}
  |f|,|\dot f|,|g|,|\dot g| < e^{\poly(\delta_C)}.
\end{equation}
Here we can and do assume for instance that we analytically continue
the leaf from $\lambda_0$ to $D^{1/2}$ along some sequence of discs in
$\P^1$ as explained in the comment following
Theorem~\ref{thm:fuchs-growth}, staying at distance
$e^{-\poly(\delta_C)}$ from the singularities. We absorb $M_0$ in the
asymptotic notation.

The estimate for $\Im (f/g)$ requires a different argument. The ratio
of periods $f/g$ defines a map $D_i^{1/4}\to\H$, and by the
Schwarz-Pick lemma we have
\begin{equation}
  \diam_H( (f/g)(D^{1/2}_i) ) \le \diam_{D^{1/4}_i} D^{1/2}_i = \const.
\end{equation}
Thus as we continue from $\lambda_0$ to $D^{1/2}$ along a finite sequence of
discs $D_i$ the ratio $f/g$ varies by at most $\poly(\delta_C)$ in
$\H$. In particular $\Im^{-1}(f/g)<e^{\poly(\delta_C)}$ in $D^{1/2}$.

The proof for the elliptic logarithms $z,w$ is similar to $f,g$. At
the origin $\lambda_1$ of $D$ we choose $z,w$ to be given by an
integral~\eqref{eq:elliptic-logarithm} with some standard choice of
the path far from $0,1,\lambda_1$. Then as before we can estimate
$|z|,|\dot z|,|w|,\dot w|$ at $\lambda_1$ by
$e^{\poly(\delta_C)}$. Our goal is to prove the same in
$D^{1/2}$. Recall that $z,w$ satisfy a non-homogeneous Picard-Fuchs
equation~\eqref{eq:logarithm-PF}. Here the right-hand side consists of
the regular functions $R_z,R_w$ on $A_\lambda$, which can be estimated
from above by $\poly(\dist^{-1}(\lambda,\Sigma))$ in the same way as
estimating the branches $P,Q$. Now using
Theorem~\ref{thm:fuchs-growth} again gives
\begin{equation}
  |z|,|\dot z|,|w|,|\dot w| < e^{\poly(\delta_C)}.
\end{equation}
To conclude, we have the following.

\begin{Lem}\label{lem:growth-estimates}
  For any $\lambda\in D^{1/2}$ we have effective estimates
  \begin{equation}
    |f|,|\dot f|,|g|,|\dot g|,|z|,|\dot z|,|w|,|\dot w|,\frac1{\Im(f/g)} \le e^{\poly(\delta_C)}.
  \end{equation}
  In other words, the ball $\cB$ constructed above is contained in
  $\M_R$ for $\log R=\poly(\delta_C)$.
\end{Lem}

\subsection{Setting up the counting}
\label{sec:counting-setup}

We will be interested in counting representations of $z,w$ as rational
combinations of $f,g$. For this it will be convenient to expand our
ambient space and foliation. Let
\begin{equation}
  \hat\M := \M \times \Mat_{2\times2}(\C)_U
\end{equation}
where $U$ denotes the coordinate on the second factor in matrix form.
We define the foliation $\hat\cF$ on $\hat\M$ as the product of the
foliation $\cF$ on $\M$ with the full-dimensional foliation on the
second factor (i.e. where a single leaf is the entire space). We will
work with a ball $\hat\cB$ of radius $\hat R$ contained in
$\hat\B_{\hat R}$, where $\hat R$ will be suitably chosen later.

Consider the subvariety $V\subset\hat\M$ given by
\begin{equation}
  V := \{ (z,w) = (f,g) U \}.
\end{equation}
Note that we do not restrict the entries of $U$ to $\R$, as this would
not be covered by Theorem~\ref{thm:alg-count}. Let $\hat\cL_0$ denote
the lifting of the standard leaf to $\hat\M$. We will apply
Theorem~\ref{thm:alg-count} with $\Phi:=U$. Let $G$ act on
$\Mat_{2\times2}(\C)_U$ on the right by on the right by the formula
\begin{equation}
  U\cdot(G_L,G_P) = G_P^{-1}(U+G_L).
\end{equation}
Then the diagonal action on $\hat\M$ restricts to an action of $G$ on
$V$, and the map $\Phi$ is of course $G$-equivariant.  We use this to
deduce two functional transcendence statements for all leafs from the
corresponding statements for the standard leaf.

\begin{Lem}\label{lem:U-finite}
  The map $\Phi\rest{\hat\cL_p\cap V}$ is finite for any $p\in\hat\M$.
\end{Lem}
\begin{proof}
  If the map is not finite then there is some $U_0$ whose fiber,
  i.e. the set
  \begin{equation}
    \{\lambda\in A_\lambda : (z(\lambda),w(\lambda))=(f(\lambda),g(\lambda))U_0\},
  \end{equation}
  is locally of dimension one. For the standard leaf $\hat\cL_0$ this
  contradicts the functional transcendence lemma
  \cite[Lemma~5.1]{mz:torsion-annalen}, as it implies $z,w$ are
  algebraic over $f,g$. Since all other leafs are obtained by the
  $G$-action, and $\Phi$ is equivariant, the same follows for all
  other leafs.
\end{proof}

\begin{Lem}\label{lem:U-blocks}
  Let $\cW\subset\C^4$ be a positive dimensional algebraic block such
  that $\Sigma(V,\cW)$ meets a ball $\cB\subset\hat\cL_0$. Then $\cW$
  is contained in the affine linear space defined by
  \begin{equation}\label{eq:U-blocks}
    (z(\lambda_0),w(\lambda_0)) =  (f(\lambda_0),g(\lambda_0)) U
  \end{equation}
  for some $\lambda_0\in\lambda(\cB)$.
\end{Lem}
\begin{proof}
  This is again just a reformulation of the functional transcendence
  results from \cite{mz:torsion-annalen}. Suppose $\cW$ is not
  contained in such an affine linear space. Then
  $\Phi(\hat\cL_0\cap V)$ contains one of the analytic components of
  (some germ of) $\cW$, and in particular $\lambda$ is non-constant on
  $\hat\cL_0\cap V$ (otherwise this germ would
  satisfy~\eqref{eq:U-blocks} for the constant value $\lambda_0$). We
  may also assume without loss of generality that $\cW$ is a curve by
  replacing it by its generic section ($\lambda$ remains non-constant
  for a generic section). Then~\eqref{eq:U-blocks} implies that
  $f(\lambda),g(\lambda)$ have transcendence degree at most $1$ over
  $z(\lambda),w(\lambda)$, contradicting
  \cite[Lemma~5.1]{mz:torsion-annalen}.
\end{proof}

We remark that Lemma~\ref{lem:U-blocks} implies, in particular, that
any block coming from the standard leaf can contain at most one real
point: it is a product of two affine-linear spaces with complex angle
$(f(\lambda_0):g(\lambda_0))$. By $G$-equivariance, the blocks coming
from other leafs are obtained as $G$-translates. For a sufficiently
nearby leaf, i.e. a $G$-translate sufficiently close to the origin,
the angle is still complex. All such nearby blocks therefore also
contain at most one real point. This will be crucial later in our
application of Theorem~\ref{thm:alg-count}.

\subsection{Finishing the proof}

We fix $\e=e^{-\poly(\delta_C)}$, to be suitably chosen later. Apply
Theorem~\ref{thm:alg-count} to the ball $\hat\cB$ with $V,\Phi$
constructed in~\secref{sec:counting-setup}. Recall that by
Lemma~\ref{lem:growth-estimates} the ball $\cB$ is contained in a ball
of radius $e^{\poly(\delta_C)}$ in $\M$. The same lemma also shows
that $\Im(f/g)\ge e^{-\poly(\delta_C)}$ uniformly on $\cB$. We choose
$\e$ small enough so that, by Lemma~\ref{lem:U-blocks}, any block
coming from a leaf of distance $\e$ to $\cB$ is still a product of
affine spaces with complex angle (and in particular contains at most
one real point). Setting $A:=\R^2\cap\Phi(\hat\cB^2\cap V)$ we have
\begin{equation}
  \# A(1,h) = \poly(\delta_C,\hat R,h).
\end{equation}

On the other hand we have the following.
\begin{Lem}
  For suitably chosen $\hat R = e^{\poly(\delta_C)}$ each Galois
  conjugate $c_\sigma$ in Lemma~\ref{lem:branch-choice} corresponds to
  a $\Q$-rational point of log-height $\poly(\delta_C,\log n)$ in $A$.
\end{Lem}
\begin{proof}
  Recall that $P(c),Q(c)$ are both torsion of order at most $n$, and
  the same is therefore true for each $c_\sigma$. In the equation
  \begin{equation}
    (z,w) = (f,g) U
  \end{equation}
  with real $U$ each $c_\sigma$ corresponds to a single value of $U$,
  with all coordinates rational and denominators not exceeding
  $n$. The claim will follows once we prove that the entries of $U$
  are bounded from above by $e^{\poly(\delta_C)}$. This follows from
  Lemma~\ref{lem:growth-estimates}. Indeed, we have for example
  \begin{equation}
    z = f u_{11} + g u_{12}
  \end{equation}
  which can be interpreted as a pair of $\R$-linear equations on
  $u_{11},u_{12}$ by taking real and imaginary parts. The determinant
  of this system is at least $e^{-\poly(\delta_C)}$ because $\Im(f/g)$
  is at least $e^{-\poly(\delta_C)}$, and the bounds on $U$ follow
  easily.
\end{proof}

In fact the proof of Theorem~\ref{thm:alg-count} gives a bound
$\poly(\delta_C,h)$ not only for $\#A(1,h)$ but for the number of
different points $\lambda\in D$ corresponding to points in $A$. A
reader having forgotten the proof of Theorem~\ref{thm:alg-count} may
instead appeal to Corollary~\ref{cor:zero-count-any-codim}, which
shows that the number of different values of $\lambda$ corresponding
to a single point of $A$ is at most $\poly(\delta_C,h)$. Indeed for
any fixed value $U=U_0$ in $A$ apply the corollary to the set
\begin{equation}
  \cB^2\cap V\cap\{U=U_0\},
\end{equation}
using Lemma~\ref{lem:U-finite} to see that $\Sigma$ is empty in this
case. It is in fact a simple exercise to remove the dependence on $h$
in this bound, but as we do not need this we leave it for the reader.

We are now ready to finish the proof. Recall that in
Lemma~\ref{lem:branch-choice} the number of points $c_\sigma$ is at
least $N(c)/\poly(\delta_C)$. Thus with $h=\poly(\delta_C,\log n)$ we
have
\begin{equation}
  N(c)/\poly(\delta_C) \le \#A(1,h) \le \poly(\delta_C,\log n)=\poly(\delta_C,\log N(c))
\end{equation}
where the last estimate is by~\eqref{eq:n-vs-Nc}. This immediately
implies $N(c)=\poly(\delta_C)$ as claimed.

\section{Proof of Theorem~\ref{thm:ao-decidable}}
\label{sec:effective-ao-proof}

\subsection{The foliation}

We follows Pila's proof \cite{pila:ao}, which employs the
uniformization of modular curves by the $j$-function $j:\Omega\to\C$
where $\Omega\subset\H$ denotes the standard fundamental domain for
the $\SL_2(\Z)$-action. To apply Theorem~\ref{thm:alg-count} we encode
this graph as a leaf of an algebraic foliation. This could be done by
replacing $j:\H\to\C$ by the $\lambda$-function $\lambda:\H\to\C$ and
expressing the inverse $\tau:\C\to\H$ as the ratio of two elliptic
integrals, which satisfy a Picard-Fuchs differential equation as
discussed in~\secref{sec:simul-tor-foliation}. For variation here we
employ an alternative approach, expressing $j$ directly as a solution
of a Schwarzian-type differential equation (which was employed for a
similar purpose in \cite{me:effective-ao}). 

Recall that the Schwarzian operator is defined by
\begin{equation}
  S(f) = \left(\frac{f''}{f'}\right)' - \frac12 \left(\frac{f''}{f'}\right)^2
\end{equation}
We introduce the differential operator
\begin{equation}
  \chi(f) = S(f) + R(f) (f')^2, \qquad R(f) = \frac{f^2-1968f+2654208}{2f^2(f-1728)^2}
\end{equation}
which is a third order algebraic differential operator vanishing on
Klein's $j$-invariant $j$ \cite[Page~20]{masser:heights}. As observed in
\cite{fs:minimality} it easy to check that the solutions of
$\chi(f)=0$ are exactly the functions of the form
$j_g(\tau):=j(g^{-1}\cdot\tau)$ where $g\in\PGL_2(\C)$ acts on $\C$ in the
standard manner.

The differential equation above may be written in the form
$f'''=A(f,f',f'')$ where $A$ is a rational function. More explicitly,
consider the ambient space $M:=\C\times\C^3\setminus\Sigma$ with
coordinates $(\tau,y,\dot y,\ddot y)$ where $\Sigma$ consists of the
zero loci of $y,y-1728$ and $\dot y$. In particular we will write
$\C_y:=\C\setminus\{0,1728\}$. On $M$ the vector field
\begin{equation}
  \xi := \pd{}\tau + \dot y \pd{}y + \ddot y\pd{}{\dot y}+A(y,\dot y,\ddot y)\pd{}{\ddot y}
\end{equation}
encodes the differential equation above, in the sense that any
trajectory is given by the graph of a function $j_g(\tau)$ and its
first two derivatives. 

We define our $n$-dimensional foliation $\cF$ on the ambient space
$\M:=M^n$ by taking an $n$-fold cartesian product of $M$ with its
one-dimensional foliation determined by the vector field $\xi$. We let
$\cL$ denote the standard leaf given by the product of the graphs of
the $j$ function, and note that any other leaf is obtained as a
product of graphs of
\begin{equation}\label{eq:j-graph-general}
  (j(g_1\tau_1),\ldots,j(g_n\tau_n)), \qquad\text{for } (g_1,\ldots,g_n)\in\GL_2(\C)^n.
\end{equation}
In fact one may easily check that $\cF$ is invariant under an
appropriate algebraic action of $\GL_2(\C)^n$, where the action is
trivial on $y$ and is computed by the chain rule on $\dot y,\ddot y$.

\subsection{Reduction to maximal special points}
\label{sec:weakly-special-reduct}

Denote by $V^\ws$ the \emph{weakly-special locus} of $V$, i.e. the
union of all weakly-special subvarieties of $V$. In
\cite[Theorem~4]{me:effective-ao} it is shown that one can effectively
compute $V^\ws$, and in particular
$\delta(V^\ws)=\poly_n(\delta_V)$. It is also shown that as a
consequence of this, one can reduce the problem of computing all
maximal special subvarieties to the problem of computing all special
points $p\in V_\alpha\setminus V_\alpha^\ws$, for some auxiliary
collection of varieties $V_\alpha\subset\C^{n_\alpha}$ with
$n_\alpha\le n$ and $\sum_\alpha \delta(V_\alpha)=\poly_n(\delta_V)$.

We remark that even though in loc. cit. only the bounds on the number
and degrees of these auxiliary subvarieties are explicitly stated, the
construction in fact yields an effective algorithm as can be observed
directly from the proof. We also note that the proof itself relies on
differential algebraic constructions, though of a very different
nature compared to the present paper. In conclusion, it will suffice
to prove Theorem~\ref{thm:ao-decidable} only for special points
outside $V^\ws$.

\subsection{A bound for maximal special points}

We will use Theorem~\ref{thm:alg-count} to count maximal special
points in $V$ as a function of the discriminant. Toward this end we
let $\hat V:=\pi_y^{-1}(V)\subset\M$ where $\pi_y:\M\to\C_y^n$ is the
projection to the coordinates $(y_1,\ldots,y_n)$. We let
$\Phi=(\tau_1,\ldots,\tau_n)$. Note that $\Phi$ restricts to the germ
of a finite map locally at every $\cL_p$.

The following corollary will allow us to control the blocks coming
from nearby leafs. We denote by $J:\H^n\to\C^n$ the $n$-fold product
of the $j$-function.

\begin{Prop}\label{prop:ws-blocks-control}
  Let $\cB$ be a $\xi$-ball in the standard leaf and $\cW$ a positive
  dimensional algebraic block coming from a nearby leaf at distance
  $\e$. Then
  \begin{equation}
    J(\cW\cap\pi_\tau(\cB)) \subset N_\delta(V^\ws), \qquad \delta=O_\cB(\e),
  \end{equation}
  where $N_\delta(V^\ws)$ denotes the $\delta$-neighborhood of $V^\ws$
  with respect to the Euclidean metric on $\C^n$.
\end{Prop}
\begin{proof}
  If $\cW$ comes from the standard leaf then the modular Ax-Lindemann
  theorem established in \cite{pila:ao} shows that $\cW$ is contained
  in a pre-weakly-special subvariety $\cW'$ with
  $\cW'\cap\H^n\subset J^{-1}(V)$. More accurately, some branch of a
  germ of $\cW$ is contained in $\cW'$, but since $\cW$ is irreducible
  in fact $\cW\subset\cW'$. Thus $J(\cW\cap\H^n)\subset V^\ws$ by
  definition.
  
  Recall that by~\eqref{eq:j-graph-general} all other leafs are
  obtained by a $g\in\GL_2(\C)^n$-translate of the standard leaf. A
  tubular neighborhood of $\cB$ of radius $\e$ is thus generated by
  translates with $\norm{\id-g}=O_\cB(\e)$. If $\cW$ comes from a leaf
  in this neighborhood then we have by the argument above
  \begin{equation}
    J(g^{-1}(\cW\cap\H^n))\subset V^\ws.
  \end{equation}
  To finish we should show that
  \begin{equation}
    J(\cW\cap\pi_\tau(\cB)) \subset N_{O_\cB(\e)}(J(g^{-1}(\cW\cap\H^n))).
  \end{equation}
  This follows at once because $\cB$ is pre-compact. First,
  $\cW\cap\pi_\tau(\cB)$ is contained in a neighborhood of
  $g^{-1}(\cW\cap\H^n)$ since the derivative of the $G$-action is
  bounded in $\pi_\tau(\cB)\subset\H^n$. And then
  $J(\cW\cap\pi_\tau(\cB))$ is contained in a neighborhood of
  $J(g^{-1}(\cW\cap\H^n))$ since the derivative of $J$ is bounded in
  $\pi_\tau(\cB)$.
\end{proof}

Let $p\in V\setminus V^\ws$ be a special point. We associate to $p$
the complexity measure
\begin{equation}
  \Delta(p) := \sum_{i=1}^n |\disc(p_i)|
\end{equation}
where $\disc p_i$ is the discriminant of the endomorphism ring of the
elliptic curve corresponding to $p_i$. The Chowla-Selberg formula
combined with standard estimates on $L$-functions implies
\begin{equation}\label{eq:cm-height-bound}
  h(p) = O_\e(\Delta(p)^{\e}), \qquad \text{for any }\e>0,
\end{equation}
see e.g. \cite[Lemma~4.1]{habegger:weakly-bounded} and the estimate
for the logarithmic derivative of the $L$-function in
\cite[Corollary~3.3]{tsimerman:ao}.

\begin{Lem}\label{lem:dist-p-Vws}
  For any $\e>0$ and special point $p\in V\setminus V^\ws$,
  \begin{equation}\label{eq:p-logdist}
    \log \dist^{-1} (p_\sigma,V^\ws) = \poly_n(\delta_V) O_\e(\Delta(p)^\e).
  \end{equation}
  holds for at least two thirds of the Galois conjugates $p_\sigma$ of
  $p$.
\end{Lem}
\begin{proof}
  This follows from $\delta(V^\ws)=\poly_n(\delta_W)$ and
  \eqref{eq:cm-height-bound}. For instance, choose a polynomial $P$
  with $h(P)=\poly_n(\delta_V)$ vanishing on $V^\ws$ but not on
  $p$. Then $h(P(p))=\poly_n(\delta_V,h(p))$ and in particular for
  two-thirds of the conjugates $p_\sigma$ we have
  \begin{align}
    -\log |p_\sigma|&=O_\e(\Delta(p)^\e) & -\log \abs{P(p_\sigma)} &= \poly_n(\delta_V,O_\e(\Delta(p)^\e)).
  \end{align}
  On the other hand, for these conjugates if
  $d_\sigma:=\dist (p_\sigma,V^\ws)$ then by the mean value theorem
  (assuming e.g. $d_\sigma<1$),
  \begin{equation}
    \abs{P(p_\sigma)} \le \d_\sigma\cdot \max_{B_{p_\sigma}(d_\sigma)} \norm{\d P} =e^{\poly_n(\delta_V,\Delta(p)^\e)} \cdot d_\sigma.
  \end{equation}
  Taking logs and comparing the last two estimates
  implies~\eqref{eq:p-logdist} on $d_\sigma$.
\end{proof}

Let $K\subset\Omega^n\subset\H^n$ be a compact subset of the
fundamental domain $\Omega^n$ with
\begin{equation}
  \vol(K) \ge \frac23 \vol(\Omega^n).
\end{equation}
According to Duke's equidistribution theorem
\cite{duke:equidistribution}, for $|\disc(p)|\gg1$ at least two-thirds
of the conjugates $p_\sigma$ correspond to points in $K$. Thus at
least one third of the conjugates $p_\sigma$ both lie in $K$ and
satisfy Lemma~\ref{lem:dist-p-Vws}. Call such conjugates \emph{good
  conjugates}.

\begin{Rem}
  Rather than appealing to equidistribution, it is also possible to
  use the height estimate~\eqref{eq:cm-height-bound} to deduce that a
  large portion of the orbit lies at log-distance at least $\Delta^\e$
  to the cusp. One can then use a logarithmic subdivision process to
  cover all such points by $\Delta^\e$-many $\vxi$-balls, similar
  to the approach we use in~\secref{sec:simul-torsion-domain}. We will
  employ such an approach in an upcoming paper (with Schmidt and
  Yafaev) on general Shimura varieties, where the analogous
  equidistribution statements are not known. 
\end{Rem}

According to Brauer-Siegel \cite{brauer-siegel} the number of good
conjugates is at least
\begin{equation}\label{eq:CM-deg-lower}
  \frac13 [\Q(p):\Q] \ge \Delta(p)^c \qquad \text{for some }c>0.
\end{equation}
We also recall from \cite{pila:ao} that for each $p_\sigma$, the
corresponding preimage $\tau_\sigma\in\Omega^n$ satisfies
\begin{align}
  [\Q(\tau_\sigma):\Q] &\le 2n & H(\tau_\sigma) &= \poly_n(\Delta(p)).
\end{align}
We are now ready to finish the proof. Cover the part of $\cL$
corresponding to $K$ by finitely many unit balls $\cB\subset\cL$ and
apply Theorem~\ref{thm:alg-count} with $\e_0$ to each of them. We
choose
\begin{equation}
  \log \e_0^{-1} = \poly_n(\delta_V) O_\e(\Delta(p)^\e)
\end{equation}
corresponding to the bound in Lemma~\ref{lem:dist-p-Vws}, so that for
any good conjugate $p_\sigma$ the $\e_0$-neighborhood of $p_\sigma$
does not meet $V^\ws$. Then according to
Corollary~\ref{prop:ws-blocks-control}, none of the positive
dimensional blocks $\cW_\alpha$ coming from nearby leafs at distance
$\e_0$ can contain the corresponding $\tau_\sigma$. Counting with $g=2n$ and
$e^h=\poly_n(\Delta(p))$ we see that each good conjugate must
come from a zero-dimensional $\cW_\alpha$, and the number of good
conjugates is therefore
$\poly_n(\delta_V,O_\e(\Delta(p)^\e))$. Choosing $\e$ sufficiently
small compared to $c$ and comparing this to~\eqref{eq:CM-deg-lower} we
conclude that $\Delta(p)<\poly_n(\delta_V)$.

\subsection{Computation of the maximal special points}

To compute the finite list of maximal special points
$p\in V\setminus V^\ws$ we start by enumerating all CM points
$p\in\C^n$ up to a given $\Delta=\delta_V$ (in polynomial time). For
example, they are all obtained as images under $\pi$ of points $\tau$
in $\H^n$, whose coordinates are each imaginary quadratic with height
$\poly_n(\Delta)$. It is simple to enumerate all such points, call
them $\{\tau_j\}$.

For each $\tau_j$ and each equation $P_k=0$ defining $V$, we should
check whether $P_k(\pi(\tau_j))$ vanishes. Since
$\delta_{\pi(\tau_j)}=\poly_n(\Delta)$ we have
\begin{equation}
  \delta(P_k(\pi(\tau_j))) = \poly_n(\Delta,\delta_V)=\poly_n(\delta_V)
\end{equation}
and by Liouville's inequality either $P_k(\pi(\tau_j))=0$ or
\begin{equation}
  -\log |P_k(\pi(\tau_j))| = \poly_n(\delta_V),
\end{equation}
so it is enough to compute $\poly_n(\delta_V)$ bits of
$P_k(\pi(\tau_j))$ to check whether it vanishes. This can be
accomplished, for instance by computing with the $q$-expansion of
$j(\cdot)$, and we leave the details for the reader.

\section{Proof of Degree bounds for torsion points}
\label{sec:degree-bounds-proof}

\subsection{Schmidt's strategy}

Our proof of Theorem~\ref{thm:torsion-deg} is based on an idea by
Schmidt \cite{schmidt:torsion-deg}, who noticed that a polylogarithmic
point-counting result such as the one obtained in
Theorem~\ref{thm:alg-count} would allow one to deduce degree bounds
for special points from suitable height bounds (in various
contexts). The idea (in the context of an abelian variety $A$) is to
count points on the graph of the universal cover $\pi:\C^g\to A$. If
$P$ is an $n$-torsion point on $A$ then one has a collection
$P,P^2,\ldots,P^n$ of torsion points. On the graph of $\pi$ these
correspond to pairs $(z_j,P^j)$ where: i) $h(P^j)$ is bounded (as
these are torsion points); ii) $h(z_j)=O(\log n)$ where we represent
$z_j$ as combinations of the periods; iii) $P^j$ all lie in the field
$\K(P)$. By point counting we therefore find
\begin{equation}
  n = \poly_A(\log n,[\K(P):\K])
\end{equation}
from which the Galois orbit lower bound follows.

Most applications of the Pila-Wilkie counting theorem use
point-counting to deduce an upper bound on the size of Galois orbits
of special points, contrasting them with lower bounds obtained by
other methods (usually transcendence techniques). Schmidt's idea shows
that polylogarithmic point counting results already carry enough
transcendence information to directly imply Galois orbit lower bounds,
giving ``purely point-counting'' proofs of unlikely intersection
statements (modulo the corresponding height bounds, which are of
course specific to the problem at hand). It is also to our knowledge
one of the first applications of point-counting that requires
polylogarithmic, rather than the classical sub-polynomial, estimates.

\begin{Rem}
  In fact for the method above to work, sub-polynomial dependence on
  the height $H:=e^h$ is sufficient. The crucial asymptotic is to
  obtain polynomial dependence on the degree $g$. However in the
  interpolation methods used to prove the Pila-Wilkie and related
  theorems, the dependence on $g$ and $h$ are of the same order.
  Imitating the proof of the classical Pila-Wilkie theorem would give
  only a sub-exponential $e^{\e g}$ bound, which is not sufficient.
\end{Rem}

\subsection{Further implications}
\label{sec:deg-bounds-future}

Though we consider here the simplest context of elliptic curves and
abelian varieties, Schmidt's idea can be made to work also in the
context of special points on Shimura varieties. In an upcoming paper
with Schmidt and Yafaev we prove that height bounds of the form
\begin{equation}\label{eq:special-height-bound}
  h(p) \ll \disc(p)^\e, \qquad \text{for any }\e>0,
\end{equation}
where $p$ is a special point in a Shimura variety and $\disc(p)$ is
the discriminant of the corresponding endomorphism ring, imply
Galois-orbit lower bounds
\begin{equation}
  [\Q(p):\Q] \ge \disc(p)^c \qquad \text{for some }c>0.
\end{equation}
In the case of the Siegel modular variety $\mathcal{A}_g$ the
bound~\eqref{eq:special-height-bound} follows from the recently
established averaged Colmez formula \cite{aghm:colmez,yz:colmez}, and
Tsimerman \cite{tsimerman:ao} has used these height bounds to
establish a corresponding Galois orbit lower bounds. For this
implication Tsimerman uses the Masser-Wustholz isogeny estimates
\cite{mw:isogeny-estimates}, another deep ingredient based on
transcendence methods. We obtain an alternative proof of Tsimerman's
theorem, avoiding the use of isogeny estimates and replacing them with
point-counting based on Theorem~\ref{thm:alg-count}. In particular our
proof applies also in the context of general Shimura varieties, where
it establishes the Andr\'e-Oort conjecture conditional on the height
bound~\eqref{eq:special-height-bound}. This seems to be of interest
because, to our knowledge, the corresponding isogeny estimates are not
known for general Shimura varieties, and it is therefore unclear
whether Tsimerman's approach could be used in this generality.

\subsection{Proof of Theorem~\ref{thm:torsion-deg}}

Write $E=E_\lambda$ in Legendre form and let
\begin{equation}
  h:=\max(h(\lambda),[\K:\Q]).
\end{equation}
It is known that that $\hFal(E)=\poly(h)$, so we prove the bound with
$h$ instead of the Faltings height. Let $\xi_E$ denote the
translation invariant vector field on $E$ given by
\begin{equation}
  \xi_E := [x(x-1)(x-\lambda)]'\partial_y+2y\partial_x.
\end{equation}
We will work in the ambient space $\M:=E_{x,y}\times\C_z$ where the
subscripts denote the coordinates used on each factor. We will
consider the foliation generated by the vector field
\begin{equation}
  \xi := \xi_E+\partial_z.
\end{equation}
Any leaf of $\cF$ is the graph of a covering map $\C\to E$, and as
usual this forms a principal $G$-bundle with $G=(\C,+)$ acting on
$\C_z$ by translation.

The main technical issue is to cover a large piece of a leaf by
$\poly(h)$-many $\xi$-balls with suitable control on the growth. For
this it is convenient to renormalize the time parametrization of
$\xi$. Recall that $x:E\to\P^1$ is ramified over the points
$\Sigma:=\{0,1,\lambda,\infty\}$. Fix some $\delta=e^{-\poly(h)}$ to
be chosen later, and denote by $\Lambda_\delta$ the complement of the
$\delta$-neighborhood of $\Sigma$. As
in~\secref{sec:simul-torsion-domain} we can choose a collection of
$\poly(h)$ discs $D_i$ such that
\begin{align}
  D_i^{1/2}&\subset\Lambda_{\delta/2}, & \Lambda_\delta&\subset\cup_i D_i.
\end{align}
We consider the reparametrized vector field $\xi':=\xi/2y$. The
$\xi'$-ball $\cB_i$ around the center of $D_i$ with the same radius
corresponds to $D_i$ in the $x$-variable and to one of the two
$y$-branches in the $y$-variable. The $z$-coordinate is obtained by
integrating $(1/2y)\d x$ over $D_i$, and since the integrand is
bounded by $e^{\poly(h)}$ we conclude the following.

\begin{Lem}\label{lem:cBi-radius}
  The $\xi'$-ball $\cB_i$ is contained in $\B_R$ for suitable
  $R=e^{\poly(h)}$.
\end{Lem}

Now let $p\in E$ be an $n$-torsion point and denote
\begin{equation}
  N(p) := [\K(p):\K].
\end{equation}
Then the Neron-Tate height of $p$ vanishes, and by Zimmer
\cite{zimmer} it follows that the usual Weil height satisfies
$h(p)=\poly(h)$. By the same arguments used to prove
Lemma~\ref{lem:orbit-distribution}, at least half of the Galois
conjugates of $p$ over $\K$, which are also $n$-torsion, have an $x$
coordinate in $\Lambda_\delta$ with some suitable choice
$\delta=e^{-\poly(h)}$.

We can apply the same argument to the points $p^2,p^3,\ldots,p^n$,
which are also torsion of order at most $n$, and which crucially
satisfy $N(p^j)\le N(p)$ since the product law is defined over
$\K$. Concluding this discussion we have the following.

\begin{Lem}\label{lem:many-pis}
  There exist at least $n/2$ points $p_i\in E$ that are: i) torsion of
  order at most $n$; ii) have height $\poly(h)$; iii) satisfy
  $x(p_i)\in\Lambda_\delta$; iv) have $N(p_i)\le N(p)$.

  At least $n/\poly(h)$ of these points have $x$-coordinate belonging
  to a single disc $D_i$ and $y$-coordinate in a fixed branch over
  $D_i$.
\end{Lem}

We will derive a contradiction to the assumption that $N(p)$ is small
by counting the points corresponding to $p_i$ on the leaf of our
foliation. Let $\tau_0\in\H$ be the element in the standard fundamental
domain corresponding to $E$, i.e. such that $E\simeq\C/\<1,\tau_0\>$. It
is known that $|\tau_0|=\poly(h)$, though even $|\tau_0|=e^{\poly(h)}$
would suffice for our purposes.

We consider the ambient space $\hat\M:=\M\times\C^2_u\times\C_\tau$
with the foliation $\hat\cF$ given by the product of $\cF$ with the
generator $\xi'$ on $\M$, the full-dimensional foliation on $\C^2_u$,
and the zero-dimensional foliation on $\C_\tau$. Consider the variety
$V\subset\hat\M$ given by
\begin{equation}
  V := \{ (x,y,z,u_1,u_2) : z=u_1+\tau u_2 \}
\end{equation}
and the map $\Phi:=(x,y,u_1,u_2)$. A leaf of $\hat\cF$ is given by
fixing a leaf of $\cF$ and a point $\tau\in\C_\tau$. Similar to
Lemma~\ref{lem:U-blocks} we have

\begin{Lem}\label{lem:U-blocks-torsion-deg}
  Let $\cW$ be a positive-dimensional algebraic block such that
  $\Sigma(V,\cW)$ meets some leaf $\hat\cL$. Then $u_1+\tau u_2$ is
  constant on $\cW$, where $\tau$ is the value taken on $\hat\cL$.
\end{Lem}
\begin{proof}
  Suppose not. Then $\cW$ would imply an algebraic relation between
  $(x,y)$ and $z=u_1+\tau u_2$ which would hold in a neighborhood of
  some point $(x,y,z)$ on a leaf $\cL$ of $\cF$. But we have seen that
  $(x,y)$ are abelian functions of $z$ (on any leaf), and are
  certainly not algebraic over $z$.
\end{proof}

Recall $R=e^{\poly(h)}$ is a constant to be chosen later.  Let
$\cB=\cB_i$ be the ball corresponding to the disc $D_i^{1/2}$ of
Lemma~\ref{lem:many-pis}. We consider the polydisc $\hat\cB$ given by
the product of $\cB$ in the $(x,y,z)$ coordinates, a polydisc of radius
$R$ in the $u_1,u_2$ coordinates, and the fixed $\tau=\tau_0$ in the
$\tau$ coordinate.  Note that $\Im\tau_0\ge1/\sqrt{2}$. Choosing $\e$
smaller than this number we deduce from
Lemma~\ref{lem:U-blocks-torsion-deg} that any block coming from a leaf
of distance $\e$ to $\hat B$ is contained in an affine line with a
complex angle in $(u_1,u_2)$ an in particular contains at most one
real point. Apply Theorem~\ref{thm:alg-count}. Then setting
$A:=\R^2\cap\Phi(\hat\cB^2\cap V)$ we have
\begin{equation}
  \# A(g,t) = \poly(h,g,t).
\end{equation}
On the other hand, we have the following.

\begin{Lem}
  Each of the points $p_i$ of Lemma~\ref{lem:many-pis} corresponds to
  a point of log-height $t=\poly(h,\log n)$ and degree at most
  $g=[\K:\Q]\cdot N(p)$ in $A$.
\end{Lem}
\begin{proof}
  For the $x,y$ coordinates this is the content of
  Lemma~\ref{lem:many-pis}. For the $u_1,u_2$ coordinates, they are
  rational with denominators at most $n$ since $p_i$ is torsion, $z$
  is a lifting of $p_i$ to $\C$, and $1,\tau_0$ generate the lattice
  of $E$. The numerators are also bounded by $e^{\poly(h)}$: for $z$
  this bound is given in Lemma~\ref{lem:cBi-radius}, and the same
  bound for $u_1,u_2\in\R$ follows since $z=u_1+\tau_0 u_2$ and
  $\Im\tau_0\ge1/\sqrt2$. Thus choosing a suitable $R=e^{\poly(h)}$ we
  see that $u_1,u_2$ are indeed rational of log-height
  $\poly(\log n,h)$ and in the polydisc of radius $R$.
\end{proof}
Finally, we have
\begin{equation}
  n/\poly(h) \le \#A(N(p)\cdot[\K:\Q],\poly(h,\log n)) = \poly(h,N(p),\log n)
\end{equation}
and it follows that $n=\poly(h,N(p))$ as claimed.

\subsection{Abelian varieties of arbitrary genus}
\label{sec:bound-any-g}

There is no difficulty in extending the proof above to show that if
$A$ is an abelian variety of genus $g$ over $\K$ and $p\in A$ is
torsion of order $n$ then $n \le \poly_A([\K(p):\K])$. The more
technically challenging part is to establish the precise dependence on
$A$, namely
\begin{equation}
  n = \poly_g([\K:\Q],[\K(p):\K],\hFal(A)).
\end{equation}
We briefly sketch how the argument presented above in the elliptic
case can be extended to arbitrary genus assuming that $A$ is
principally polarized.

An explicit embedding of $A$ in projective space can be computed in
terms of theta function, $\Theta:A\to\P^N$. The \emph{theta height} of
$A$ is defined by $h:=h_\Theta(A)=h(\Theta(0))$. By
\cite[Corollary~1.3]{pazuki:theta} the Faltings height is roughly the
same as the theta height, and we can use this as a replacement of
$h(\lambda)$ used in the elliptic case. By
e.g. \cite[Lemma~3.1]{mw:periods-minimal} the image $\Theta(A)$ is
defined by a collection of quadratic equations whose coefficients are
functions of $\Theta(0)$, so as in the elliptic case we have
\begin{equation}
  h(\Phi(A))=\poly_g(h).
\end{equation}
The translation-invariant vector fields $\vxi:=(\xi_1,\ldots,\xi_g)$
used to construct the foliation can also be explicitly expressed in
terms of $h(\Theta(0))$ \cite[Lemma~3.7]{mw:periods-minimal}, and in
particular $\delta_\vxi=\poly_g(h)$.

The main technical issue is the covering of $A$ by $\poly_g(h)$-many
$\vxi$-balls. (Here if one is content with a general bound depending
on $A$ rather than polynomial in $h$, then compactness can be
used). In the elliptic case we achieved this by explicitly
constructing a covering by balls in the $x$-coordinate. In arbitrary
dimension one obviously needs a more systematic approach. For
instance, the results of \cite{me:complex-cells} show that $\Theta(A)$
can be covered by $\const(g)$ charts whose domains are \emph{complex
  cells}. When $\Theta(A)$ is further assumed to be of height $h$ one
can in fact replace these general cells by $\poly_g(h)$ polydiscs
(this is a work in progress with Novikov and Zack). Having obtained
such a collection of polydiscs replacing our discs $D_i$ in the
elliptic case, one can proceed with the proof without major changes.

\appendix

\section{Growth estimates for inhomogeneous Fuchsian equations}
\label{appendix:fuchs-growth}

\subsection{Gronwall for higher-order linear ODEs}

Let $D\subset\C$ be a disc and consider a linear differential operator
\begin{equation}
  L = a_0(t) \partial_t^n+a_1(t)\partial_t^{n-1}+\cdots+a_n(t)
\end{equation}
where $a_0,\ldots,a_n$ are holomorphic in $\bar D$. Let $b(t)$ also be
holomorphic in $\bar D$. We will consider the growth of solutions for
the inhomogeneous equation
\begin{equation}\label{eq:inhomogeneous-eq}
  Lf=b(t).
\end{equation}
We denote
\begin{align}
  j_t^nf&:=(f,\partial_tf,\ldots,\partial_t^nf)^T & v_b&:=(0,\ldots,0,b(t))^T.
\end{align}
The following is a form of the Gronwall inequality for monic linear
operators.

\begin{Lem}\label{lem:scalar-gronwall}
  Suppose that $a_0\equiv1$ and denote
  \begin{align}
    A &= \max_{j=1,\ldots,n}\max_{t\in\bar D} |a_j(t)|, & B &= \max_{t\in\bar D} |b(t)|.
  \end{align}
  Then for every $t\in D$,
  \begin{equation}
    \norm{j_t^nf(t)} \le e^{O_n(A)} (O_n(B)+\norm{j_t^nf(0)}).
  \end{equation}
\end{Lem}
\begin{proof}
  Rewrite $Lf=b$ as a linear system for the vector $j_n^tf$ as follows
  \begin{multline}
    \partial_t j_t^nf(t) =
    \begin{pmatrix}
      0 & 1 & 0 & \cdots &  0 \\
      0 & 0 & 1 & \cdots &  0  \\
      & & \vdots & & \\
      -a_n(t) & -a_{n-1}(t) & \cdots & -a_2(t) & -a_1(t)
    \end{pmatrix}
    j_t^nf(t) + v_b(t) =\\
    \Omega(t)j_t^nf(t)+v_b(t)
  \end{multline}
  Then for $t\in D$ the solution $j_t^nf$ satisfies
  \begin{equation}
    \partial_t \norm{j_t^nf(t)} \le \norm{\Omega}\cdot\norm{j_t^nf(t)}+O_n(B) = O_n(A)\norm{j^nf_t(t)}+O_n(B)
  \end{equation}
  and the conclusion follows by the classical Gronwall's inequality.
\end{proof}

Lemma~\ref{lem:scalar-gronwall} allows one to prove growth estimates
for general equations $Lf=b$ non-singular in a disc $D$ by first
dividing by the leading term. However, due to the exponential
dependence on $A$, the resulting bound will grow exponentially as a
function of the minimum of the leading term. For arbitrary singular
linear ODEs this is the best one can expect.

For Fuchsian operators, which are the operators that come up in the
study of periods and logarithms, one can obtain much sharper
estimates with polynomial growth near the singularities. We do this
in the following section.

\subsection{Inhomogeneous Fuchsian equations}
  
In this section we assume that the coefficients of $L$ are in
$\C[t]$. Recall that $L$ is called \emph{Fuchsian} is each singular
point $t_0\in\P^1$ of $L$ is Fuchsian. This means that in a local
coordinate $z$ where the $t_0$ is the origin, $L$ can be written in
the form
\begin{equation}
  L = \tilde a_0(z)(z\partial_z)^n+\tilde a_1(z)(z\partial_z)^{n-1}+\cdots+\tilde a_n(z)
\end{equation}
where the coefficients $\tilde a_j$ are holomorphic at the origin, and
$\tilde a_0(0)\neq0$. We denote by $\Sigma\subset\P^1$ the set of
singular points of $L$.

We recall the notion of \emph{slope} for a differential operator over
$\C(t)$ introduced in \cite{me:inf16}. For a polynomial $p$ we define
$\norm{p}$ to be the $\ell_1$-norm on the coefficients. We extends
this to rational functions by setting $\norm{p/q}=\norm{p}/\norm{q}$
where the fraction $p/q$ is reduced.

\begin{Def}[Slope of a differential operator]
  The \emph{slope} of $\angle L$ of $L$ is defined by
  \begin{equation}
    \angle L := \max_{i=1,\ldots,n} \frac{\norm{a_j(t)}}{\norm{a_0(t)}}.
  \end{equation}
  The \emph{invariant slope} $\slope L$ is defined by
  \begin{equation}
    \slope L := \sup_{\phi\in\Aut(\P^1)} \angle (\phi^*L)
  \end{equation}
  where $\phi^*L$ denote the pullback of $L$ by $\phi$.
\end{Def}

We remark that in \cite{me:inf16} the slope was defined by first
normalizing the coefficients $a_j$ to be polynomials, but this minor
technical difference does not affect what follows. It is a general
fact that the invariant slope is finite for Fuchsian operators
\cite[Proposition~32]{me:inf16}. The following gives effective
estimates when $L$ is defined over a number field $\K$. In this case
we denote $\delta_L:=\sum_j\delta_{a_j}$.

\begin{Prop}\label{prop:slope-bound}
  Suppose $L$ is defined over a number field. Then
  $\slope L=e^{\poly_n(\delta_L)}$.  
\end{Prop}
\begin{proof}
  Since $\slope L$ is defined by a semialgebraic formula over a number
  field and is known to be finite, an effective bound follows from
  general effective semialgebraic geometry \cite{basu:bounds}, see the
  derivation in \cite[Section~3.6.2]{me:inf16}.
\end{proof} 

The slope $\slope L$ is useful for the study of oscillation of
solutions of homogeneous Fuchsian equations $Lf=0$, and is similarly
useful for the study of growth. In the inhomogeneous case we also
require the following corollary concerning the leading coefficient.
We denote by $a_j(L)$ the $j$-th coefficient of $L$.

\begin{Prop}\label{prop:leading-term-bound}
  Suppose $L$ is defined over a number field. Then
  \begin{equation}
    \inf_{\phi\in\Aut(\P^1)} \norm{a_0(\phi^*L)} = e^{-\poly_n(\delta_L)}.
  \end{equation}
\end{Prop}
\begin{proof}
  We first prove that the infimum is positive. Assume the
  contrary. Then we may choose $\phi$ such that $\norm{a_0(\phi^*L)}$
  is arbitrarily small. By boundedness of $\slope L$ this means that
  $\norm{a_j(\phi^*L)}$ is also arbitrarily small. Now the operator
  $L':=L+1$ is also Fuchsian, and
  $\norm{a_0(\phi^*L')}=\norm{a_0(\phi^*L)}$ is arbitrarily small
  while $\norm{a_n(\phi^*L')}=\norm{1+a_n(\phi^*L)}$ is arbitrarily
  close to $1$. This contradicts the boundedness of $\slope L'$. The
  effective bound is then obtained in the same way as in
  Proposition~\ref{prop:slope-bound}.
\end{proof}

We will also need the following simple lemma.

\begin{Lem}\label{lem:rational-norm-bound}
  Let $r$ be a rational function, and $D$ denote the unit disc. If $r$
  has no poles in $D^{1/2}$ then
  \begin{equation}
    \max_{z\in D}|r(z)| \le e^{O(\deg r)} \norm{r},
  \end{equation}
  and if $r$ has no zeros in $D^{1/2}$ then
  \begin{equation}
     \min_{z\in D}|r(z)| \ge e^{-O(\deg r)}\norm{r}.
  \end{equation}
\end{Lem}
\begin{proof}
  Without loss of generality $\norm{r}=1$. Write $r=p/q$ with $p,q$
  polynomials and $\norm{p}=\norm{q}=1$. Suppose $q$ has no zeros in
  $D^{1/2}$. Then
  \begin{equation}
    |q(z)| \ge e^{-O(\deg q)} \text{ for every }z\in D
  \end{equation}
  by e.g. \cite[Lemma~7]{me:polyfuchs}. Since $|p(z)|$ is bounded by
  $1$ for $z\in D$, the upper bound on $r(z)$ follows. The lower bound
  follows by repeating the above for $1/r$.
\end{proof}

We now come to our main theorem. Below if $D=D_r(t_0)$ is a disc we
call $z=(t-t_0)/r$ a natural coordinate on $D$.

\begin{Thm}\label{thm:fuchs-growth}
  Let $L$ be a Fuchsian operator as above, defined over a number
  field. Let $D=D_r(t_0)$ be a disc with
  $D^{1/2}\subset\C\setminus\Sigma$ and $z$ a natural coordinate on
  $D$. Consider the equation $Lf=b$ where $b$ is defined in $D^{1/2}$
  and bounded by $B$ there. Then for $t_1\in D$,
  \begin{equation}
    \norm{j^n_zf(t_1)} \le e^{c_L} \big(B+\norm{j^n_zf(t_0)}\big), \qquad c_L:=e^{\poly_n(\delta_L)}.
  \end{equation}
  In particular
  \begin{equation}
    \norm{j^n_tf(t_1)} \le \max(r,1/r)^n e^{c_L} \big(B+\norm{j^n_tf(t_0)}\big).
  \end{equation}
\end{Thm}
\begin{proof}
  Note that $j^n_zf$ is obtained from $j^n_tf$ by
  multiplying the $j$-th coordinate by $r^j$, so the second estimate
  follows from the first.
  
  Let $\hat L$ denote the pullback of $L$ to the $z$-coordinate and
  set $\hat a_j=a_j(\hat L)$. By
  Propositions~\ref{prop:slope-bound}~and~\ref{prop:leading-term-bound},
  we have
  \begin{equation}\label{eq:hatL-bounds}
    \angle \hat L, \norm{\hat a_0}^{-1} = e^{\poly(\delta_L)}.
  \end{equation}
  Dividing by the leading term we have an equation
  \begin{equation}
    (\partial_z^n+\frac{\hat a_1}{\hat a_0}\partial_z^{n-1}+\cdots+\frac{\hat a_n}{\hat a_0})f=b/\hat a_0.
  \end{equation}
  The claim will now follow from Lemma~\ref{lem:scalar-gronwall} once
  we establish suitable bounds for the coefficients and for the right
  hand side. These bounds follow from~\eqref{eq:hatL-bounds} and
  Lemma~\ref{lem:rational-norm-bound} applied to obtain a lower bound
  for $\hat a_0$ (which has no zeros in $D^{1/2}$) and an upper bound
  for $\hat a_j$ (which has no poles in $D^{1/2}$).
\end{proof}

Theorem~\ref{thm:fuchs-growth} allows one to obtain a polynomial bound
on the growth of solutions for equations $Lf=b$, assuming $b$ has
polynomial growth. To see this consider a fixed $t_0\in\C$ and an
arbitrary $t_1\in\C$, say of distance $\delta$ to $\Sigma$. Connect
$t_0$ to $t_1$ by a sequence of $O(\log\delta)$ discs $D_i$ with
$D_i^{1/2}\subset\C\setminus\Sigma$ such that the sequence of radii is
$r_i$ satisfies e.g. $1/2<r_i/r_{i+1}<2$. It is a simple exercise in
plane geometry to check that this can always be achieved. Then
applying Theorem~\ref{thm:fuchs-growth} consecutively for the discs
$D_i$, and assuming $b$ is bounded by $\poly(1/\delta)$ throughout
gives an estimate on the branch of $f$ at $t_1$ obtained by analytic
continuation along the $D_i$, namely
\begin{equation}
  f(t_1) = \poly_L(1/\delta)\norm{j^n_tf(t_0)}.
\end{equation}
Here one should use the statement in the natural coordinate $z$,
noting that by our assumption on the radii the distortion in jets when
switching from coordinate $z_i$ to $z_{i+1}$ is bounded by $2^n$ at
each step. If one uses the statement with the $t$-coordinate then one
gets the slightly larger $\delta^{O(\log\delta)}$ term (which is still
suitable for our purposes in this paper).

\begin{Rem}
  The geometric requirements on the chains of discs $D_i$ are not
  arbitrary, they represent an actual obstruction. For instance,
  consider the function
  \begin{equation}
    f(x) = \sqrt{\e^2+x^2}+x.
  \end{equation}
  As an algebraic function, this satisfies a Fuchsian equation
  $L_\e f=0$ with singularities at $\{\e,-\e,\infty\}$. For $\e\ll1$,
  one branch of this function becomes uniformly small while the other
  tends uniformly to $2x$. On the other hand the slope of the
  operators $L_\e$ is uniformly bounded as a function of $\e$, for
  instance by the results of \cite{me:inf16} (or by direct computation
  for this simple case). However, to analytically continue from one of
  these branches to the other, one must at some point pass between
  $-\e$ and $\e$. To do this some of the discs $D_i$ would have to be
  of size $O(\e)$, and this explains why one cannot obtain an estimate
  for one branch in terms of the other branch which is uniform in
  $\e$.
\end{Rem}

\bibliographystyle{plain} \bibliography{nrefs}

\end{document}